\documentclass[preprint,12pt]{elsarticle}
\usepackage{amsmath,amssymb,amsthm}
\usepackage[usenames,dvipsnames]{color}
\usepackage{enumerate}
\usepackage{geometry}
\usepackage{graphicx}
\usepackage{fullpage}

\numberwithin{equation}{section}
\newtheorem{theorem}[equation]{Theorem}

\newtheorem{lemma}[equation]{Lemma}
\theoremstyle{definition}
\newtheorem{remark}[equation]{Remark}

\newtheorem{assumption}[equation]{Assumption}

\newcommand{\dt}{{\Delta t}}
\newcommand{\Plh}{I_H}

\newcommand{\R}{\mathbb{R}}

\usepackage[backgroundcolor=gray!30,linecolor=black]{todonotes}

\graphicspath{{/Users/iliescu/papers/papers-2019/da-rom-2018/latex/backup/ROM_DA_2_20_19/}}

\usepackage{url}

\begin{document}
	
\begin{frontmatter}

\title{Continuous Data Assimilation Reduced Order Models of Fluid Flow}

\author{Camille Zerfas\fnref{fn1}}
\ead{czerfas@clemson.edu}
\address{Department of Mathematical Sciences, Clemson University, Clemson, SC, 29634}
\author{Leo G. Rebholz\fnref{fn1}}
\ead{rebholz@clemson.edu}
\address{Department of Mathematical Sciences, Clemson University, Clemson, SC, 29634}
\author{Michael Schneier\corref{cor1}}
\ead{mhs64@pitt.edu}
\address{Department of Mathematics, University of Pittsburgh, Pittsburgh, PA, 15213}
\author{Traian Iliescu\fnref{fn2}}
\ead{iliescu@vt.edu}
\address{Department of Mathematics, Virginia Tech, Blacksburg, VA, 24061}

\cortext[cor1]{Corresponding author}
\fntext[fn1]{This author was partially supported by NSF Grant DMS 1522191.}
\fntext[fn2]{This author was partially supported by NSF Grant DMS 1821145.}



\date{}

\begin{abstract}
We propose, analyze, and test a novel continuous data assimilation reduced order model (DA-ROM) for simulating incompressible flows.  While ROMs have a long history of success on certain problems with recurring dominant structures, they tend to lose accuracy on more complicated problems and over longer time intervals.  Meanwhile, continuous data assimilation (DA) has recently been used to improve accuracy and, in particular, long time accuracy in fluid simulations by incorporating measurement data into the simulation.  This paper synthesizes these two ideas, in an attempt to address inaccuracies in ROM by applying DA, especially over long time intervals and when only inaccurate snapshots are available.  We prove that with a properly chosen nudging parameter, the proposed DA-ROM algorithm converges exponentially fast in time to the true solution, up to discretization and ROM truncation errors.  Finally, we propose a strategy for nudging adaptively in time, by adjusting dissipation arising from the nudging term to better match true solution energy.  Numerical tests confirm all results, and show that the DA-ROM strategy with adaptive nudging can be highly effective at providing long time accuracy in ROMs.
\end{abstract}

\begin{keyword}
Navier-Stokes equations, proper orthogonal decomposition, data assimilation, reduced order modeling.
	
\end{keyword}

\end{frontmatter}


\section{Introduction}


{\it Reduced order models (ROMs)} for fluids dominated by relatively few recurrent spatial structures are generally built as follows~\cite{hesthaven2015certified,HLB96,quarteroni2015reduced}: 
(i) postulate a collection of snapshots, either from numerical experiments or from physical data;
(ii) from those snapshots, select a small number (e.g., $10$) of ROM basis functions; 
(iii) project the equations of motion into this basis; and 
(iv) advance the velocity in time to interrogate flows different from the one generating the snapshots. 
ROMs have been explored for decades~\cite{HLB96}.
When successful, ROMs can decrease the computational cost of a brute force, direct numerical simulation (DNS) by orders of magnitude.

One of the main roadblocks for ROMs of realistic flows is their {\it lack of accuracy}, e.g., in complex problems, for long time intervals, or when a low-dimensional ROM basis is used.
To increase the ROM accuracy in practical applications, several approaches are currently used.
We list some of these below:

{\it (i) Closure Modeling}: \ 
To model the effect of the discarded ROM modes, a {\it Correction} term is generally added to the standard ROM~\cite{benosman2017learning,osth2014need,protas2015optimal,san2015stabilized}.
Given the drastic truncation used in ROMs for realistic flows, the Correction term is essential for accuracy.

{\it (ii) Numerical Stabilization}: \  
To eliminate/alleviate the spurious numerical oscillations generated when ROMs are used for convection-dominated flows, numerical stabilization techniques can be used~\cite{carlberg2017galerkin,gunzburger2018leray,san2015stabilized}.

{\it (iii) Data-Driven Modeling}: \ 
Recently, available numerical or experimental data have been used to construct ROM operators~\cite{peherstorfer2016data} or to determine the unknown coefficients in classical ROM operators~\cite{benosman2017learning,galletti2004low,protas2015optimal}.
 
{\it (iv) Improved Basis}: \  
Another approach for increasing the ROM accuracy in practical applications is the construction of an improved (more accurate) ROM basis that better captures the behavior of the underlying system~\cite{balajewicz2016minimal,rowley2004model,san2015principal,xie2018lagrangian}.

{\it (v) Physical Accuracy}: \  
To develop physically sound ROMs, recent effort has been directed at ensuring that the ROMs {\it  satisfy the same physical balances/conservation laws} as those satisfied by the equations of motion~\cite{mohebujjaman2017energy,mohebujjaman2019physically}.

In this paper, we propose a new approach to increase the ROM accuracy.
Specifically, we use {\it data assimilation (DA)} to develop a novel DA-ROM.
In weather modeling, climate science, and hydrological and environmental forecasting, DA has been used for decades to incorporate observational data in simulations, in order to increase the accuracy of solutions and to obtain better estimates of initial conditions~\cite{kalnay2003atmospheric}.
In this paper, we use DA to improve the ROM accuracy.
Specifically, we add to the standard ROM a feedback control term of the form
\begin{eqnarray}
	\mu \, I_{H} (u_{r} - u_{obs}) ,
	\label{eqn:nudging-term}
\end{eqnarray}
which nudges the ROM approximation ($u_{r}$) towards the reference solution ($u_{obs}$) corresponding to the observed data.
In~\eqref{eqn:nudging-term}, $I_{H}$ is an interpolation operator onto a coarser mesh of size $H$ and $\mu > 0$ is a nudging parameter.
Equation~\eqref{eqn:nudging-term} allows the {\it simple} implementation of DA into existing ROM codes.
The nudging term~\eqref{eqn:nudging-term} increases the accuracy of the new DA-ROM by utilizing the available low-resolution data, without the need to increase the number of ROM basis functions. This type of data assimilation has recently become popular due to a seminal paper of Azouani et al. \cite{Azouani_Olson_Titi_2014}, and since then it has been used to improve solutions to many different types of evolutionary systems \cite{Bessaih_Olson_Titi_2015,Biswas_Martinez_2017,Farhat_Jolly_Titi_2015,Farhat_Lunasin_Titi_2016abridged,Foias_Mondaini_Titi_2016,Jolly_Martinez_Titi_2017,Larios_Pei_2017_KSE_DA_NL,Markowich_Titi_Trabelsi_2016}.

We emphasize that our new DA-ROM is different from other uses of DA for ROMs, e.g.,~\cite{cao2007reduced,kaercher2018reduced,maday2015parameterized,stefanescu2015pod}:
In the latter the authors use ROMs to {\it speed up classical DA algorithms} (e.g., 4D-VAR), whereas in the DA-ROM proposed in this paper, we use DA to {\it improve the ROM accuracy}. 

The rest of the paper is organized as follows:
In Section~\ref{sec:notation}, we introduce some notation and preliminaries necessary for our analysis.
In Section~\ref{sec:da-rom-analysis}, we construct the new DA-ROM and perform a careful error analysis.
In Section~\ref{sec:numerical-results}, we perform a numerical investigation of the new DA-ROM in the numerical simulation of a 2D flow past a circular cylinder and  discuss implementation of the DA-ROM algorithm with an adaptive nudging parameter, which can be used to further improve the accuracy of solutions. 
Finally, in Section~\ref{sec:conclusions}, we draw conclusions and outline future research directions.

\section{Notation and Preliminaries}
	\label{sec:notation}

Let  $\Omega \subset \R^d$, $d$=2 or 3, be a bounded open domain. The $L^2(\Omega)$ norm and inner product will be denoted by $\| \cdot \|$ and $(\cdot, \cdot)$, respectively, and all other norms will be appropriately labeled with subscripts. 

We consider the {\it Navier-Stokes equations (NSE)} with no-slip boundary conditions:
\begin{equation}\label{eqn:nse-1}
\begin{aligned}
&u_t + u\cdot\nabla u + \nabla p - \nu\Delta u =  f,\ \text{and } \nabla \cdot u =  0,\  \text{in} \ \Omega \times (0,T]   \\
&u   =  0, \ \text{on} \ \partial\Omega \times (0,T], \ \text{and } u(x,0)  =  u_0(x), \ \text{in} \ \Omega.  \\
\end{aligned}
\end{equation}
Here $u$ is the velocity, $f = f(x,t)$ is the known body force, $p$ is the pressure, and $\nu$ is the kinematic viscosity.

We denote the natural velocity space by $X=H^1_0(\Omega)$ and pressure space by $Q = L_0^2(\Omega)$ and by $(X_{h}, Q_h) \subset (X,Q)$, corresponding inf-sup stable finite element spaces. Additionally, we define the  discretely divergence-free space $V_h$ as
$$
V_{h} :=\{v_{h}\in X_{h}\,:\,(\nabla\cdot v_{h},q_{h})=0\,\,\forall
q_{h}\in Q_h\}  \subset X.
$$  The Poincar\'e inequality will be used throughout this paper: there exists a constant $C_P$ depending only on $\Omega$ such that 
\[
\| \phi \| \le C_P \| \nabla \phi \| \ \forall \phi \in X.
\]
We define the trilinear form
$$
b(w,u,v) = (w\cdot\nabla u,v) 
\qquad\forall u,v,w\in X,
$$
and the explicitly skew-symmetric trilinear form given by 
$$
b^{\ast}(w,u,v):=\frac{1}{2}(w\cdot\nabla u,v)-\frac{1}{2}(w\cdot\nabla v,u)
\qquad\forall u,v,w\in X \, .
$$
An important property of the $b^{\ast}$ operator is that $b^{\ast}(u,v,v)=0$ for $u,v\in X$. We will utilize the following bounds on the operator $b^*$ \cite{layton2008introduction}.
\begin{lemma}\label{bbounds}
	There exists a constant $M>0$ dependent only on $\Omega$ satisfying
	\begin{align*}
	|b^*(u,v,w)| & \le M \| \nabla u \| \| \nabla v \| \| \nabla w\|, \\
	|b^*(u,v,w)| & \le M \| u \|^{1/2} \| \nabla u \|^{1/2} \| \nabla v \| \| \nabla w\|, \\
	|b^*(u,v,w)| & \le M \| u \| (\| \nabla v \|_{L^3} + \| v \|_{L^{\infty}} ) \| \nabla w\|,
	\end{align*}
	for all $u,v,w\in X$ for which the norms on the right hand sides are finite.
\end{lemma}
The following lemma is proven in \cite{LRZ18}, and is useful in our analysis.

\begin{lemma}\label{geoseries}
	Suppose constants $r$ and $B$ satisfy $r>1$, $B\ge 0$.  Then if the sequence of real numbers $\{a_n\}$ satisfies 
	\begin{align}
	ra_{n+1} \le a_n + B, \label{sequence}
	\end{align}
	we have that
	\[
	a_{n+1} \le a_0\left(\frac{1}{r}\right)^{n+1}  + \frac{B }{r-1}.
	\]
\end{lemma}


\subsection{ROM preliminaries}

Let $\{ u_h^1,..., u_h^M  \}$ be snapshots of FE solutions at $M$ different time instances. The proper orthogonal decomposition seeks a low-dimensional basis that approximates these snapshots optimally with respect to a certain norm; in this paper, we use the $L^2$ norm. This minimization can be set up as an eigenvalue problem $YY^TM_h \varphi_j = \lambda_j\varphi_j$, $j = 1,...,N_h$. where $N_{h}$ is the dimension of the finite element space. The eigenvalues are real and non-negative, so they can be ordered as $\lambda_1\geq ...\geq \lambda_d \geq \lambda_{d+1} =... = \lambda_{N_h} = 0$, where $d$ is the rank of the snapshot matrix. We take the ROM space to be $X_r :=\text{span} \{ \varphi_i  \}_{i = 1}^r$, and note that $X_r \subset V_{h}$. The ROM approximation of the velocity is defined as 
\[ u_r(x,t) = \sum_{j = 1}^{r}a_j(t) \varphi_j(x),  \]
where the coefficients $a_j(t)$ are determined by solving the Galerkin ROM:
\[  (u_{r,t} , \varphi_i) + \nu (\nabla u_r, \nabla \varphi_i) + b^*(u_r, u_r, \varphi_i) = (f, \varphi_i).  \]

We define the $L^2$ ROM projection $P_r : L^2 \to X_r$ by: for all $v \in L^2(\Omega)$, $P_r(v)$ is the unique element of $X_r$ such that 
\begin{equation}\label{ROM-projection}
(P_r(v), v_r) = (v,v_r) \ \ \  \forall \ v_r \in X_r.  \end{equation}
In addition, the following inverse inequality holds for our ROM basis \cite{KV01}.
\begin{lemma}[POD inverse estimate]
	\begin{equation}\label{POD:inveq}
	\|\nabla \varphi \| \leq |||\mathbb{S}_{R}|||_{2}^{1/2}\|\varphi\| \ \ \ \forall \varphi \in X_{r},
	\end{equation}
	where $|||\mathbb{S}_R|||$ is the matrix 2-norm of the ROM stiffness matrix, as in Lemma 3.1 of \cite{ iliescu2014variational}.
\end{lemma}
In order to establish an error estimate for the ROM projection, we first make the following assumption on the finite element error:  

\begin{assumption}\label{assump-FE}
	Let $C(\nu,p)$ denote a constant which is dependent upon the viscosity and pressure. We assume that the finite element error $u_h$ satisfies the following error estimate
	\begin{equation}
	\begin{aligned}
	\|u^{M} - u^{M}_{h} \|^{2} + \nu h^2\dt \sum_{n=1}^{M} \|\nabla(u^{n}-u^{n}_{h})\|^{2} &\leq C(\nu,p)(h^{2k + 2} + \Delta t^{4}).
	\end{aligned}
	\end{equation}
\end{assumption}

\begin{remark}
	Error estimates of this form have been proven for varying amounts of regularity on the continuous solution $u$ and $p$. Some examples include the scheme used in the numerical experiments in Section 4.
\end{remark}

Using Assumption \ref{assump-FE} the following error estimates for the ROM projection can be proven \cite{iliescu2014variational}:
\begin{lemma} \label{ROMlemma}
	The $L^2$ ROM projection of $u^n$ satisfies the following error estimates:
	\begin{align}
	\sum_{n=1}^{M}\|u^n - P_r(u^n)\|^2 &\leq C(\nu,p)\bigg( h^{2k+ 2} + \dt ^4 + \sum_{j = r+1}^{d} \lambda_j\bigg), \label{projerr1}
	\\ \sum_{n=1}^{M} \|\nabla (u^n - P_r(u^n))\|^2 &\leq {C(\nu,p)}\bigg(h^{2k} + |||{\mathbb S}_R|||_{2} h^{2k + 2}   + (1+|||{\mathbb S}_R|||_{2})\Delta t^4 \nonumber 
	\\ & \hspace{4cm} + \sum_{j=r+1}^{d} \|  \nabla {\varphi}_j\| ^2\lambda_j\bigg). \label{projerr2}
	\end{align}
\end{lemma}

We then make the following assumption similar to that made in \cite{iliescu2014variational}:

\begin{assumption}\label{assumption:rom-projection}
	The $L^2$ ROM projection of $u^n$ satisfies the following error estimates:
	\begin{align}
	\max_{n} \|u^n - P_r(u^n)\|^2 &\leq C(\nu,p) \bigg((h^{2k+ 2} + \dt ^4) + \sum_{j = r+1}^{d} \lambda_j\bigg), \label{assump_projerr1}
	\\ \max_{n} \|\nabla (u^n - P_r(u^n))\|^2 &\leq  {C(\nu,p)}\bigg(h^{2k} + |||{\mathbb S}_R|||_{2} h^{2k + 2}   + (1+|||{\mathbb S}_R|||_{2})\Delta t^4 \nonumber 
	\\ & \hspace{4cm} + \sum_{j=r+1}^{d} \|  \nabla {\varphi}_j\| ^2\lambda_j\bigg). \label{assump_projerr2}
	\end{align}
\end{assumption}

\begin{remark}
	If we assumed in Assumption \ref{assump-FE} that the finite element error satisfies 
	\begin{equation*}
	\|u^{M} - u^{M}_{h} \|^{2} + h^{2}\|\nabla(u^{n}-u^{n}_{h})\|^{2} \leq C(\nu,p)(h^{2k + 2} + \Delta t^{4}),
	\end{equation*}
	then the bound in Assumption \ref{assumption:rom-projection} would hold. Error estimates of this form have been proven for varying amounts of regularity on the continuous solution $u$ and $p$. Some examples include the incremental pressure correction schemes in \cite{G99} and chapter 7 of \cite{P97}.
\end{remark}

%


\subsection{Data assimilation preliminaries}

We consider $I_H$ to be an interpolation operator that satisfies: For a given mesh $\tau_H(\Omega)$ with $H\le 1$,
\begin{align}
\| I_H  (w) - w \| &\le C_{I} H \| \nabla w\|, \label{interp1}
\\ \| I_H(w) \| &\leq  C_{I}\|w \|, \label{interp2}
\end{align}
for any $w\in H^1(\Omega)$.  For example, this holds for the $L^2$ projection onto piecewise constants, and the Scott-Zhang interpolant.  For the (unknown) true solution $u$, $I_H(u)$ represents an approximation of what is observed of the true solution.  We assume in this paper that $I_H(u)$ can be observed at any time.  We remark that in the nudging used in this paper, we use the interpolant of the test function as well, as suggested in \cite{RX18}.  This does not affect the convergence analysis; some extra terms arise that can be handled without difficulty, but it allows for unconditional stability.

\section{Error Analysis}
\label{sec:da-rom-analysis}
For simplicity of exposition, our analysis considers a first order DA-ROM algorithm, which takes the following form: Find $u_r^{n+1} \in X_r$ such that for all $v_r \in X_r$, 
\begin{align}
\frac{1}{\dt} (u_r^{n+1} - u_r^n, v_r) + b^*(u_r^{n+1},& u_r^{n+1}, v_r) + \nu (\nabla u_r^{n+1}, \nabla v_r)     \nonumber 
\\ &  +  \mu (\Plh (u_r^{n+1} - u(t^{n+1})), \Plh v_r) = (f^{n+1}, v_r), \label{DAROMalg}
\end{align} 
for $n = 1,2,..., M$, with $v_0=P_r(u_0)$, and where $\mu\ge 0$ is the nudging parameter.  Extension to other time stepping methods is possible, and, for example, extension to BDF2 can be done following the usual techniques {\cite{LRZ18}.  All of our numerical tests use the analogous BDF2 algorithm.
	
	We first prove a stability estimate for the DA ROM algorithm. 
	\begin{lemma}\label{lemma:stability}
		The solutions to \eqref{DAROMalg} satisfy for all $M>1$, 
		\begin{align*}
		\|u_r^{M}\|^2 
		\leq \|u_r^0 \|^2 \left( \frac{1}{1 + \lambda \dt}  \right)^{M} + C \lambda^{-1}(\nu^{-1}F^2 + \mu U^2) := C_{data},
		\end{align*}
		where $F := \|f\|_{L^\infty(0,\infty; H^{-1})}$, $U := \|u\|_{L^\infty(0,\infty ; L^2)}$, and $\lambda = \nu C_P^{-2}$.
	\end{lemma}
	\begin{proof}
		This result follows as in \cite{LRZ18} by letting $v_r = u_r^{n+1}$ in \eqref{DAROMalg} and using Cauchy-Schwarz and Young's inequalities. Additionally, the non-negative DA term $\|\Plh u_r^{n+1}\|^2$ can be dropped from the left after bounding the right hand side.  
	\end{proof}

	To analyze rates of convergence of the approximation we make the following regularity assumptions on the NSE \cite{layton2008introduction}:
	\begin{assumption}\label{assumption:reg}
		We assume that the solution of the NSE satisfies 
		\begin{equation*}
		\begin{aligned}
		&u \in L^{\infty}(0,T;H^{1}(\Omega))\cap H^{1}(0,T;H^{k+1}(\Omega))\cap
		H^{2}(0,T;H^{1}(\Omega)), \\
		&p \in L^{2}(0,T; H^{k+1}(\Omega)), \\
		&f \in L^{2}(0,T;L^{2}(\Omega)).
		\end{aligned}
		\end{equation*}
	\end{assumption}
	
	We next prove that solutions to \eqref{DAROMalg} converge to the true solution exponentially fast, up to discretization and ROM projection error. 
	\begin{theorem} \label{convtheorem}
		Define 
		\begin{align*}
			\alpha_1 &:= {\nu} - 2\mu (\beta_{2}-1)C^{2}_{I}H^{2}, 
			\\ \alpha_2 &:= 2\mu - \frac{\mu C^{2}_{I}}{2\beta_{1}} - \frac{\mu}{2\beta_{2}} - 6\nu^{-1}M^{2}||| \mathbb{S}_{R} |||_{2} \|\nabla u^{n+1}\|^{2},
			\end{align*}
			which have parameters $\mu$, $H$, $\beta_{i}>0$, $i=1,2$ that are chosen so that $\alpha_i >0$, $i = 1,2$. Then under the  regularity assumptions of Assumption \ref{assumption:reg}, we have that
		\begin{equation}
		\begin{aligned}
		\|u^{n+1}& - u_r^{n+1} \|^2  \leq \|u^{0} - u_r^{0} \|^2 \left( \frac{1}{1 + 2\lambda \dt}  \right)^{n+1} 
		\\
		&+C\lambda^{-1}\bigg\{ \dt^2 +  \nu^{-1} h^{2k} +  \beta_{1}C^{2}_{I}  \mu \bigg(h^{2k+2} + \dt^4 + \sum_{j = r+1}^{d} \lambda_j \bigg)
		\\ &+ ( \nu^{-1}M^2 +  \nu^{-1}M^2|||\mathbb{S}_{R}|||_{2}) \bigg(h^{2k} + \dt^4 + \sum_{j = r+1}^{d} \|\nabla \varphi_j\|^2\lambda_j \bigg)\bigg\},
		\end{aligned}
		\label{convresult}
		\end{equation}
		where $\lambda = \min\{ \alpha_1 C_P^{-2}, \alpha_2 \}$.
	\end{theorem}
\begin{remark}
	The $\dt^2$ term that shows up on the right hand side of \eqref{convresult} is a result of the first order time stepping in Algorithm \ref{DAROMalg}. If we instead used a second order approximation, like BDF2, then this term would be replaced by $\dt^4$. 
\end{remark}
\begin{remark}
	If $\beta_1, \beta_2$ are chosen to be $1/2$, the condition $\alpha_1 >0$ reduces to $\nu - C\mu H^2 >0$, which is the same condition found in \cite{LRZ18} and references therein, for a relationship between the nudging parameter, viscosity, and coarse mesh width. Choosing $\beta_1 , \beta_2$ larger can  allow one to choose the coarse mesh width $H$ larger (and thus require less observational data) while still satisfying $\alpha_i >0$, $i=1,2$. However, there is a trade-off because $\beta_1$ appears on the right hand side of equation \eqref{convresult}: As $\beta_1$ increases, so does the bound on the DA-ROM error. 
\end{remark}
	\begin{proof}
		The NSE (true) solution satisfies
		\begin{align}
		\frac{1}{\dt}(u^{n+1}- u^n, v_r) + b^*(u^{n+1}, & u^{n+1}, v_r) + \nu(\nabla u^{n+1}, \nabla v_r) + (p^{n+1}, \nabla \cdot v_r) \nonumber
		\\ & = (f^{n+1}, v_r) + \left( \frac{1}{\dt}(u^{n+1} - u^n) - u_t^{n+1}, v_r\right).
		\label{truesoln}
		\end{align}
		Note that we can write the time derivative term above as $C \dt u_{tt}(t^*)$ for some $t^* \in (t^n, t^{n+1})$ \cite{LRZ18}. Subtracting \eqref{DAROMalg} from \eqref{truesoln} and letting $e^n := u_r^{n} - u^n$, we obtain 
		\begin{align}
		\frac{1}{\dt}(e^{n+1} & - e^n, v_r)  + \nu(\nabla e^{n+1}, \nabla v_r) + \mu (\Plh e^{n+1}, \Plh v_r) \nonumber 
		\\ & \leq C \dt \left( u_{tt}(t^*), v_r\right) + b^*(u_r^{n+1}, e^{n+1}, v_r) + b^*(e^{n+1}, u^{n+1}, v_r)  + (p^{n+1}, \nabla \cdot v_r).
		\label{diffeqn}
		\end{align}
		Decompose the error as a part inside the ROM space and one outside by adding and subtracting the $L^2$ projection of $u^n$ into the ROM space, $P_r(u^n)$ (see \eqref{ROM-projection}):
		\[  e^{n} = (u_r^n - P_r(u^n)) + (P_r(u^n) - u^n) = : \phi_r^{n} + \eta^n.  \]
		 Letting $v_r = \phi_r^{n+1}$ in \eqref{diffeqn}, we note that since $\phi_r^{n+1} \in X_{r} \subset V_{h}$, for any $q_h\in Q_{h}$,
		\begin{align}
		(p^{n+1}, \nabla \cdot \phi_r^{n+1}) = (p^{n+1} - q_{h}, \nabla \cdot \phi_r^{n+1}).
		\end{align}
		Adding and subtracting $\phi_r^{n+1}$ to both components of the nudging term we have
		\begin{equation}
		\begin{aligned}
		&(\Plh \phi_r^{n+1} + \Plh \eta^{n+1} + \phi_r^{n+1} - \phi_r^{n+1}, \Plh \phi_r^{n+1} + \phi_r^{n+1} - \phi_r^{n+1})
		\\
		&=\| \phi_r^{n+1}\|^{2} + (\phi_r^{n+1}, \Plh \phi_r^{n+1}  - \phi_r^{n+1}) + (\Plh \phi_r^{n+1} + \Plh \eta^{n+1}  - \phi_r^{n+1}, \Plh \phi_r^{n+1} + \phi_r^{n+1} - \phi_r^{n+1})
		\\
		& = \| \phi_r^{n+1}\|^{2} + (\phi_r^{n+1}, \Plh \phi_r^{n+1}  - \phi_r^{n+1}) +(\Plh \eta^{n+1}  , \Plh \phi_r^{n+1} + \phi_r^{n+1} - \phi_r^{n+1})  
		\\
		&+ (\Plh \phi_r^{n+1} - \phi_r^{n+1}, \Plh \phi_r^{n+1} - \phi_r^{n+1} + \phi_r^{n+1})
		\\
		&= \| \phi_r^{n+1}\|^{2} + 2(\phi_r^{n+1}, \Plh \phi_r^{n+1}  - \phi_r^{n+1}) + (\Plh \eta^{n+1}  , \Plh \phi_r^{n+1}) +  \|\Plh \phi_r^{n+1} - \phi_r^{n+1}\|^{2}.  
		\end{aligned}
		\end{equation}
		Using the polarization identity, the fact that $(\eta^{n+1}-\eta^{n},\phi^{n+1}_{r}) = 0$ (by the definition of the $L^{2}$ projection), and dropping the nonnegative term $\frac{1}{2 \dt}\|\phi_r^{n+1} - \phi_r^{n}\|^{2}$ on the left hand side, we have
		\begin{equation}\label{diffeqn2}
		\begin{aligned}
		\frac{1}{2\dt}& [\|\phi_r^{n+1}\|^2 - \|\phi_r^n\|^2 ]  + \nu\|\nabla \phi_r^{n+1}\|^2  + \mu \|\phi_r^{n+1}\|^2  + \mu \|\Plh\phi_r^{n+1} - \phi_r^{n+1}\|^2
		\\ & \leq \nu | (\nabla \eta^{n+1} , \nabla \phi_r^{n+1})  + \left|C \dt \left( u_{tt}(t^*), \phi^{n+1}_r\right)\right| 
		+ | b^*(u_r^{n+1}, \eta^{n+1}, \phi^{n+1}_r)|  
		\\ & \,\,\,\, +  |b^*(\eta^{n+1}, u^{n+1}, \phi^{n+1}_r) | 	+ |b^*(\phi_r^{n+1}, u^{n+1}, \phi^{n+1}_r) |  
		+ |(p^{n+1} - q_{h}, \nabla \cdot \phi_r^{n+1}) |   
		\\ & \,\,\,\, + \mu | (\Plh \eta^{n+1}, \Plh \phi_r^{n+1}) | 
		+ 2\mu |(\phi_r^{n+1}, \Plh\phi_r^{n+1} - \phi_r^{n+1})|.
		\end{aligned}
		\end{equation}
		By Poincar\'{e}, Cauchy Schwarz, and Young's inequalities, we bound the first two terms on the right hand side and the pressure term,
		\begin{equation}
		\begin{aligned}
		\nu (\nabla \eta^{n+1},\nabla \phi^{n+1}_{r}) &\leq \frac{\nu}{4c_{1}} \|\nabla \eta^{n+1}\|^{2} + c_{1}\nu \|\nabla \phi_{r}^{n+1}\|^{2} ,
		\\
		C \dt \left( u_{tt}(t^*), \phi^{n+1}_r\right) &\leq \frac{C \dt^{2} \nu^{-1}}{4c_{2}}\|u_{tt}(t^*)\|^{2} + c_{2}\nu\| \nabla \phi_{r}^{n+1}\|^{2},
		\\
		(p^{n+1} - q_{h},\nabla \cdot \phi_{r}^{n+1}) &\leq \frac{\nu^{-1}}{4c_{3}}\|p^{n+1} - q_{h}^{n+1}\|^{2} + c_{3}\nu\| \nabla \phi_{r}^{n+1}\|^{2}.
		\end{aligned}
		\end{equation}
		The first two nonlinear terms are now bounded similarly to those in \cite{mohebujjaman2017energy} using Cauchy-Schwarz and Young's inequalities, and the first inequality from Lemma \ref{bbounds}:
		\begin{align}
		&b^*(\eta^{n+1}, u^{n+1}, \phi^{n+1}_r) \leq \frac{\nu^{-1} M^2}{4c_{4}} \|\nabla u^{n+1}\|^2\|\nabla \eta^{n+1}\|^2 +  c_{4}\nu\|\nabla \phi^{n+1}_r\|^{2},
		\\
		& b^*(u_r^{n+1}, \eta^{n+1}, \phi^{n+1}_r) \leq \frac{\nu^{-1} M^2}{4 c_{5}} \|\nabla u_{r}^{n+1}\|^2\|\nabla \eta^{n+1}\|^2 +  c_{5}\nu\|\nabla \phi^{n+1}_r\|^{2}.
		\end{align}
		How we treat the third nonlinear term  is the key difference in the proof from standard schemes (see chapter 9 of \cite{layton2008introduction}). Due to the added dissipation from the DA term on the left-hand side of \eqref{diffeqn2}, we will be able to hide the term containing $\phi^{n+1}_r$, rather than invoking a discrete Gronwall's inequality. Thus, for this term we use the second inequality from Lemma \ref{bbounds} and the ROM inverse inequality \eqref{POD:inveq} to obtain
		\begin{equation}
		\begin{aligned}
		b^*(\phi^{n+1}_r, u^{n+1}, \phi^{n+1}_r) &\leq M \|\phi^{n+1}_{r}\|^{1/2}\|\nabla\phi^{n+1}_{r}\|^{1/2} \|\nabla u^{n+1} \|\|\nabla\phi^{n+1}_{r}\|
		\\
		&\leq \frac{\nu^{-1}M^{2}||| \mathbb{S}_{R} |||_{2}}{4 c_{6}} \|\nabla u^{n+1}\|^{2}\|\phi^{n+1}_{r}\|^{2} + c_{6}\nu\|\nabla \phi_{r}^{n+1}\|^{2}.
		\end{aligned}
		\end{equation} 
		The first nudging terms on the right hand side of \eqref{diffeqn2} are bounded using \eqref{interp2}, Cauchy Schwarz, and Young's inequality
		\begin{equation}
		\begin{aligned}
		\mu (\Plh \eta^{n+1}, \Plh \phi_r^{n+1}) &\leq \frac{\mu}{4\beta_{1}}\|\Plh \phi_r^{n+1}\|^{2} + \mu\beta_{1} \|\Plh \eta^{n+1}\|^{2} 
		\\
		&\leq \frac{\mu C^{2}_{I}}{2\beta_{1}}\|\phi_r^{n+1}\|^{2} + 2\mu \beta_{1} C^{2}_{I}\|\eta^{n+1}\|^{2}.
		\end{aligned}
		\end{equation}
		The second nudging term is bounded using Cauchy Schwarz and Young's inequality, and \eqref{interp1}, yielding 
		\begin{equation}
		\begin{aligned}
		2\mu (\phi_r^{n+1}, \Plh\phi_r^{n+1} - \phi_r^{n+1}) &\leq \frac{\mu}{4\beta_{2}}\|\phi_r^{n+1}\|^{2} + \mu\beta_{2}\|\Plh\phi_r^{n+1} - \phi_r^{n+1}\|^{2}
		\\
		&= \frac{\mu}{4\beta_{2}}\|\phi_r^{n+1}\|^{2} + \mu(\beta_{2}-1)\|\Plh\phi_r^{n+1} - \phi_r^{n+1}\|^{2} + \mu\|\Plh\phi_r^{n+1} - \phi_r^{n+1}\|^{2}
		\\
		&\leq  \frac{\mu}{4\beta_{2}}\|\phi_r^{n+1}\|^{2} + C^{2}_{I}H^{2}\mu(\beta_{2}-1)\|\nabla \phi_r^{n+1}\|^{2} + \mu\|\Plh\phi_r^{n+1} - \phi_r^{n+1}\|^{2} .
		\end{aligned}
		\end{equation}
		Now letting $c_i = \frac{1}{12}$, $i=1,2,..,6$, combining terms, and recalling our definition of $\alpha_1$ and $\alpha_2$ given in the statement of the theorem, \eqref{diffeqn2} becomes
		\begin{equation}\label{diffeqn3}
		\begin{aligned}
		\|&\phi_r^{n+1}\|^2 + \alpha_1 \dt \|\nabla \phi_r^{n+1}\|^2 + \alpha_2 \dt \|\phi_r^{n+1}\|^2 
		\\ & \leq \|\phi_r^{n}\|^2
		+ C\dt^3\nu^{-1} \|u_{tt}(t^*)\|^2 
		+ C\dt \nu^{-1} M^2 \|\nabla \eta^{n+1}\|^2 \|\nabla u^{n+1}\|^2
		\\ &+ C\dt \nu^{-1} M^2 \|\nabla \eta^{n+1}\|^2 \|\nabla u_{r}^{n+1}\|^2 +C \nu \dt \|\nabla \eta^{n+1}\|^2 + C \nu^{-1} \dt \|p^{n+1} - q_h\|^2
		\\
		&+ 2C^{2}_{I}\beta_{1} \dt \mu   \|\eta^{n+1}\|^2 ,
		\end{aligned}
		\end{equation}
		where $C$ is a generic constant which is independent of $\nu,p,u,T,H,C_{I}$. Next, we bound the fourth term on the right hand side further using the ROM inverse inequality \eqref{POD:inveq}, and the stability result from Lemma \ref{lemma:stability}
		\begin{align}
		C\dt \nu^{-1} M^2 \|\nabla \eta^{n+1}\|^2 \|\nabla u_{r}^{n+1}\|^2 \leq C C_{data}\dt \nu^{-1} M^2 |||\mathbb{S}_{R}|||_{2}\|\nabla \eta^{n+1}\|^2 .
		\end{align}
		Now applying Lemma \ref{ROMlemma}, using our regularity assumptions, and taking $\lambda := \min\{ C_P^{-2} \alpha_1, \alpha_2 \}$ in \eqref{diffeqn3}, it then follows that
		\begin{equation}
		\begin{aligned}
		(1 + &2\lambda\dt )\|\phi_r^{n+1} \|^2 
		\\ & \leq \|\phi_r^{n} \|^2 + C \dt^3 + C\nu^{-1}\dt h^{2k} +  2C^{2}_{I}\beta_{1} \dt \mu \bigg(h^{2k + 2} + \dt^4 + \sum_{j = r+1}^{d} \lambda_j \bigg)
		\\ & \,\,\,\,+ \dt(C \nu^{-1}M^2 + C C_{data} \nu^{-1}M^2|||\mathbb{S}_{R}|||_{2} + C\nu) \bigg(h^{2k} + \dt^4 + \sum_{j = r+1}^{d} \|\nabla \varphi_j\|^2\lambda_j \bigg).
		\end{aligned}
		\end{equation}
		Finally, by Lemma \ref{geoseries}, we obtain
		\begin{equation}
		\begin{aligned}
		\|&\phi_r^{n+1} \|^2  \leq \|\phi_r^{0} \|^2 \left( \frac{1}{1 + 2\lambda \dt}  \right)^{n+1} \\
		&+2\lambda^{-1}\dt^{-1}\bigg\{C \dt^3 + C \nu^{-1}\dt h^{2k} +  2\beta_{1}C^{2}_{I} \dt \mu \bigg(h^{2k+2} + \dt^4 + \sum_{j = r+1}^{d} \lambda_j \bigg)
		\\ &+ \dt(C \nu^{-1}M^2 + C C_{data} \nu^{-1}M^2|||\mathbb{S}_{R}|||_{2} + C\nu) \bigg(h^{2k} + \dt^4 + \sum_{j = r+1}^{d} \|\nabla \varphi_j\|^2\lambda_j \bigg)\bigg\}.
		\end{aligned}
		\end{equation}
		The triangle inequality completes the proof.

	\end{proof}

	\section{Numerical Experiments}
	\label{sec:numerical-results}
	
	\begin{figure}[!ht]
		\centering
		\includegraphics[scale = .8]{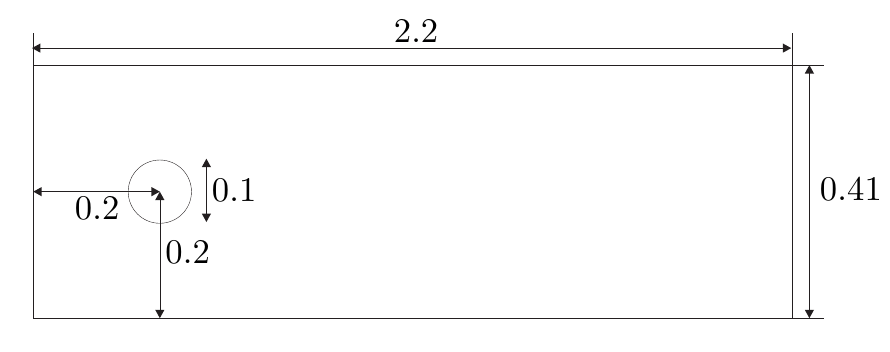}
		\caption{\label{cylinderdomain} Shown above is the domain for the flow past a cylinder test problem.}
	\end{figure}
	
In this section, we perform a numerical investigation of the new DA-ROM.
In Section~\ref{sec:rates-of-convergence}, we illustrate the theoretical scalings proved in Section~\ref{sec:da-rom-analysis}.
In Section~\ref{sec:numerical-accuracy}, we investigate the numerical accuracy of the new DA-ROM.
In Section~\ref{sec:inaccurate-snapshots}, we investigate the new DA-ROM when inaccurate snapshots are used in its construction.
Finally, in  Section~\ref{sec:variable-param}, we propose and investigate an adaptive nudging procedure.

We consider Algorithm \ref{DAROMalg} (except here with BDF2) applied to 2D channel flow past a cylinder \cite{ST96}, with Reynolds number $Re$=500.  The domain is the rectangular channel  {\color{black}[0, 2.2]}$\times$[0, 0.41], with a cylinder centered at $(0.2,0.2)$ and radius $0.05,$ see Figure \ref{cylinderdomain}. There is no external forcing ($f=0$), no-slip boundary conditions are prescribed for the walls and the cylinder, and an inflow profile is given by 
	\begin{align*}
	u_1(0,y,t) & = u_1(2.2,y,t) = \frac{6}{0.41^2}y(0.41-y),
	\\u_2(0,y,t) & = u_2(2.2,y,t) = 0.
	\end{align*}
	We take $\nu=0.0002$, and enforce the zero-traction boundary condition with the usual `do-nothing' condition at the outflow.  
	
	The DNS is run to t=15 with the usual BDF2-FEM discretization \cite{layton2008introduction} using $(P_2,P_1^{disc})$ Scott-Vogelius elements on a barycenter refined Delaunay mesh that provided 103K velocity dof, a time step of $\Delta t=0.002$, and with the simulations starting from rest ($v_h^0=v_h^{1}=0$).  Lift and drag calculations were performed for the computed solution and compared to the literature \cite{ST96,XMRT18}, which verified the accuracy of the {\color{black}DNS}.  We used the snapshots from t=5 to t=6 to generate the ROM modes.
	
	For the lift and drag calculations, we used the definitions
	\begin{align*}
	c_d(t) &= 20\int_S \left( \nu \frac{\partial u_{t_S}(t)}{\partial n}n_y - p(t)n_x \right)dS,
	\\c_l(t) &= 20 \int_S \left( \nu \frac{\partial u_{t_S}(t)}{\partial n}n_x - p(t)n_y \right)dS,
	\end{align*}
	where $u_{t_S}$ is tangential velocity, $S$ is the cylinder, and $n = \langle n_x, n_y\rangle$ the outward unit normal to the domain. For the calculations, we used the global integral formulation from \cite{J04}.
	
	The coarse mesh for DA is constructed using the intersection of a uniform rectangular mesh with the domain.  We take $H$ to be the width of each rectangle, and use $H=\frac{2.2}{20}$ (400 measurement locations) in our tests.  Figure \ref{cylmesh} shows in red a 35K dof mesh and associated $H=\frac{2.2}{8}$ coarse mesh in black.
	
	\begin{figure}[!ht]
		\centering
		\includegraphics[width = .50\textwidth,height=.12\textwidth]{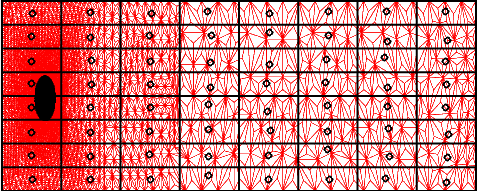}	
		\caption{\label{cylmesh} Shown above is a FE mesh (in red) and the $H=\frac{2.2}{8}$ coarse mesh and nodes (in black).}
	\end{figure}
	
	For the DA-ROM computations, we start from zero initial conditions $v_h^1 = v_h^0=0$, use the same spatial and temporal discretization parameters as the DNS, and start assimilation with the $t=5$ DNS solution (i.e., time 0 for DA-ROM corresponds to $t=5$ for the DNS).   

\subsection{Rates of Convergence}
	\label{sec:rates-of-convergence}

In this section, we illustrate numerically the rates of convergence in Section~\ref{sec:da-rom-analysis}.
Theorem \ref{convtheorem} gave a DA-ROM error estimate that depends on the ROM eigenvalues and eigenvectors, for sufficiently large $n$ and assumptions on $\mu$ and $H$:
\[  
\|u^{n+1} - u_r^{n+1}\| \leq C(\nu)\left(  \dt^2 + h^{k+1} + \left( \sum_{j=r+1}^d \lambda_j(1+\|\nabla \varphi_j\|^2)\right)^{1/2}  \right), 
\]
where $(\lambda_j, \varphi_j)$ are the eigenpairs of the ROM eigenvalue problem described in Section 2.1. Table \ref{error_table} illustrates the dependence of the error bound on the dimension of the DA-ROM space, $r$.  Taking $T$=1, $\mu=100$, $H = \frac{2.2}{20}$, $Re$=500, we run the ROM with varying $r$ and calculate the $L^2$ spatial error at the last time step.  We also calculate the quantity in the error estimate corresponding to the eigenvalues and eigenvectors (i.e., $\sum_{j=r+1}^d \lambda_j(1+\|\nabla \varphi_j\|^2))^{1/2}$), and use this and the error to calculate the corresponding convergence rate with respect to increasing $r$.  From the theorem, we expect a rate of 1, and our results are consistent with this rate.

\begin{table}[h!]
	\centering
	\begin{tabular}{|c|c|c|c|}
		\hline
		No. modes	&	$ ( \sum_{j=r+1}^d \lambda_j(1+\|\nabla \varphi_j\|^2))^{1/2} $	&	Error	&	Rate	\\ \hline 
		8	&	2.218e+2	&	4.980e-2	& --	\\
		10	&	1.077e+2	&	4.850e-2	&	1.74	\\
		12	&	9.246e+1	&	3.046e-2  &	2.51	\\
		14	&	7.680e+1	&	1.793e-2	&	1.70	\\
		16	&	4.590e+1	&	1.360e-2	   &	1.36	\\
		18	&	3.334e+1	&	9.498e-3 	&	1.12	\\
		20	&	2.601e+1	&	6.974e-3	&	1.24	\\ \hline
	\end{tabular}
	\caption{
		DA-ROM rates of convergence with respect to the ROM truncation.  
		\label{error_table}
	}
\end{table}

\subsection{Numerical Accuracy}
	\label{sec:numerical-accuracy}

In this section, we investigate the numerical accuracy of the new DA-ROM.  
Specifically, we compare the performance of the DA-ROM to that of the standard ROM ($\mu=0$) and the DNS solution in predicting energy and drag (lift is accurate in all of our tests, so we omit it here).  We run to t=10, and run tests with both $N=8$ and $N=16$ modes, and with varying $\mu = 0, 10, 100$ (we also ran $\mu=1$, and results are very close to those for $\mu=0$).  Results are shown in figure \ref{N8_re500_plots} for energy and drag prediction, and we observe a big improvement from using DA.  For $N=16$ and $\mu=100$, very good accuracy is achieved from the DA-ROM.  For $N=8$, $\mu=10$ is somewhat more accurate than for $\mu=100$, but both are better than no DA.

\begin{figure}[!ht]
	\centering
	$N=8$ \\
	\includegraphics[width = .69\textwidth,height=.27\textwidth]{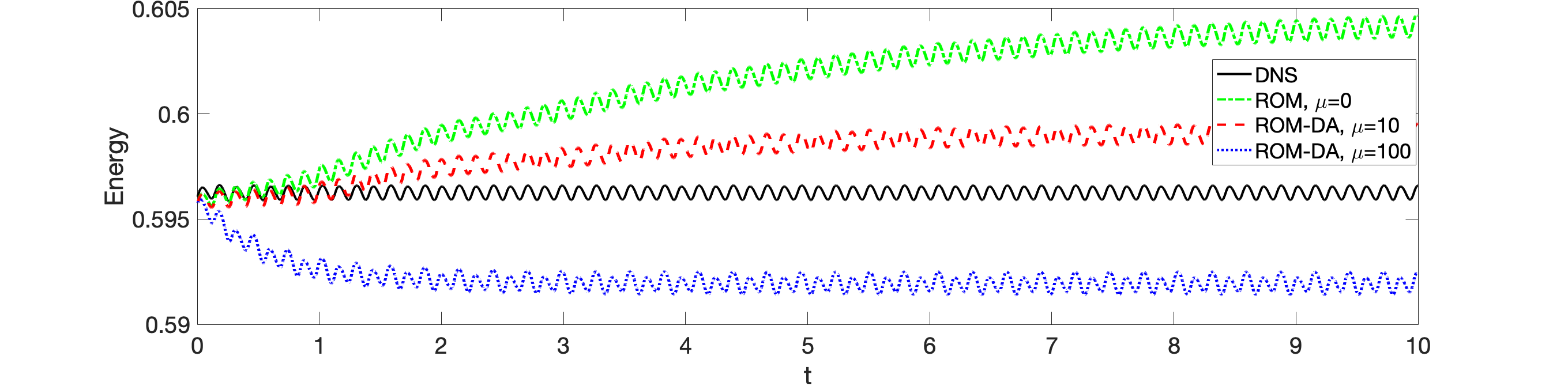}	
	\includegraphics[width = .3\textwidth,height=.27\textwidth]{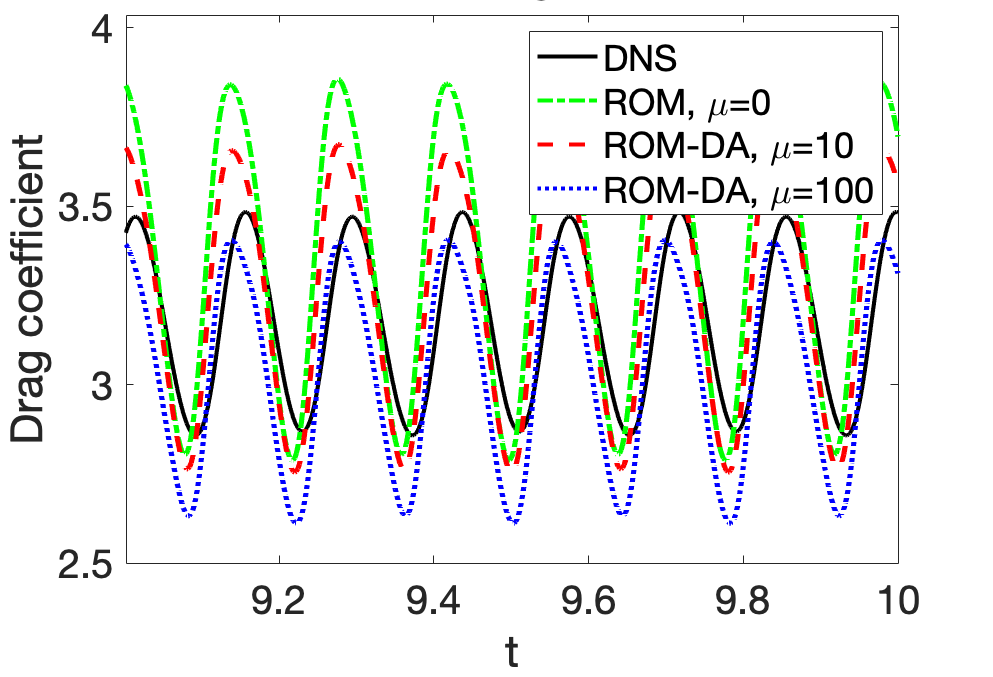}	 \\
	$N=16$ \\
	\includegraphics[width = .68\textwidth, height=.27\textwidth]{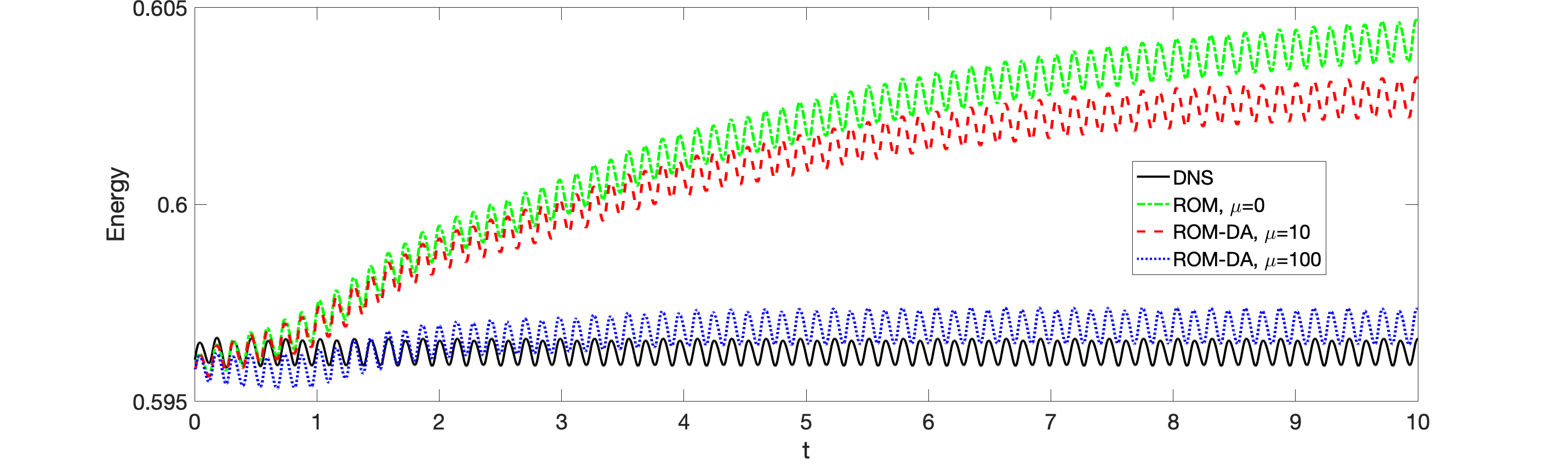}	 
	\includegraphics[width = .3\textwidth,height=.27\textwidth]{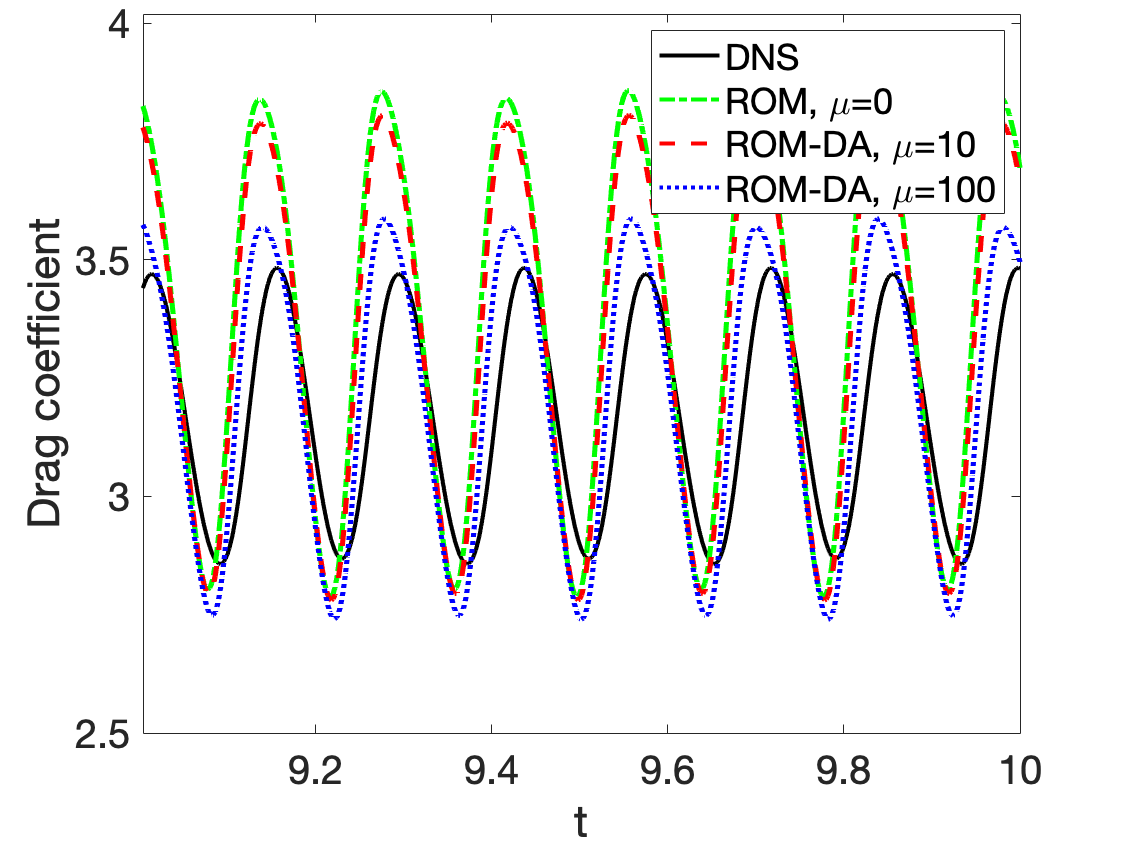}
	\caption{\label{N8_re500_plots} Shown above are the energy and drag coefficient versus time for  $Re=500$ DA-ROM with different choices of $\mu$,  $H=\frac{2.2}{20}$, and with 8 modes (top) and 16 modes (bottom).}
\end{figure}


\subsection{Inaccurate Snapshots}
	\label{sec:inaccurate-snapshots}

\begin{figure}[!ht]
	\centering
	Full basis: \hspace{1.4in} Basis 1: \hspace{1.4in} Basis 2: \\
	\includegraphics[width = 0.32\textwidth]{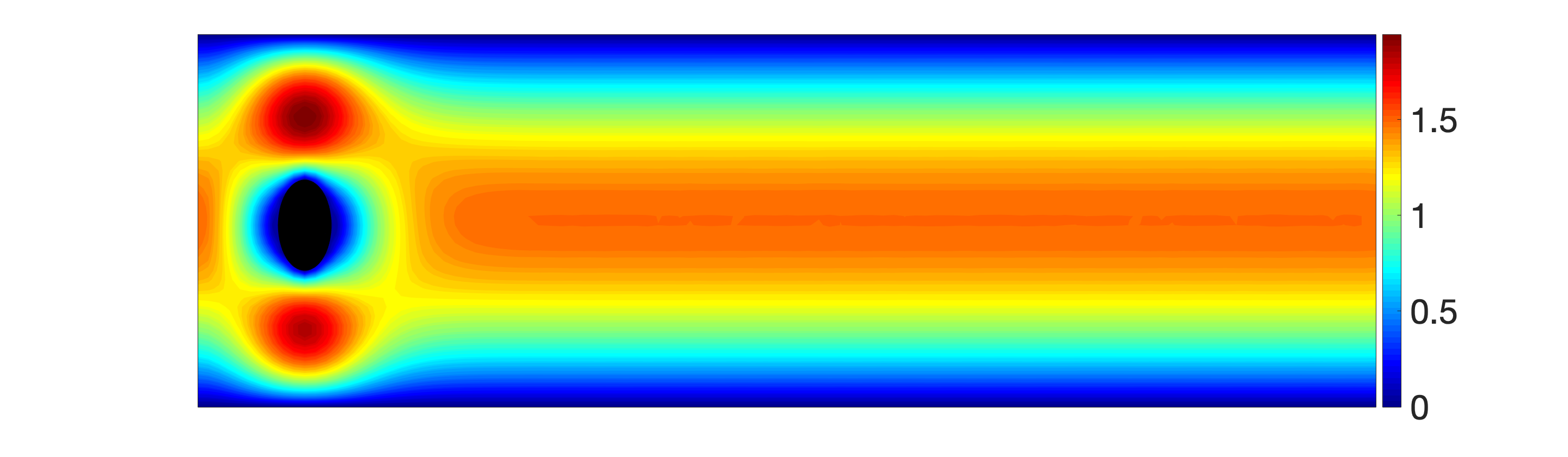}
	\includegraphics[width = 0.32\textwidth]{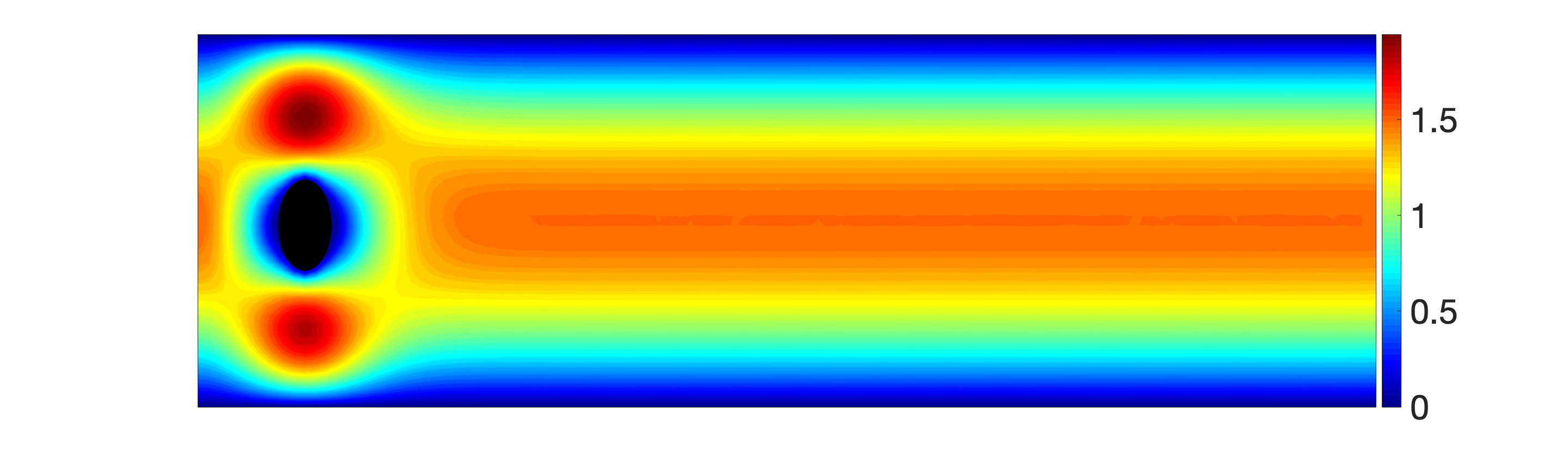}	 
	\includegraphics[width = 0.32\textwidth]{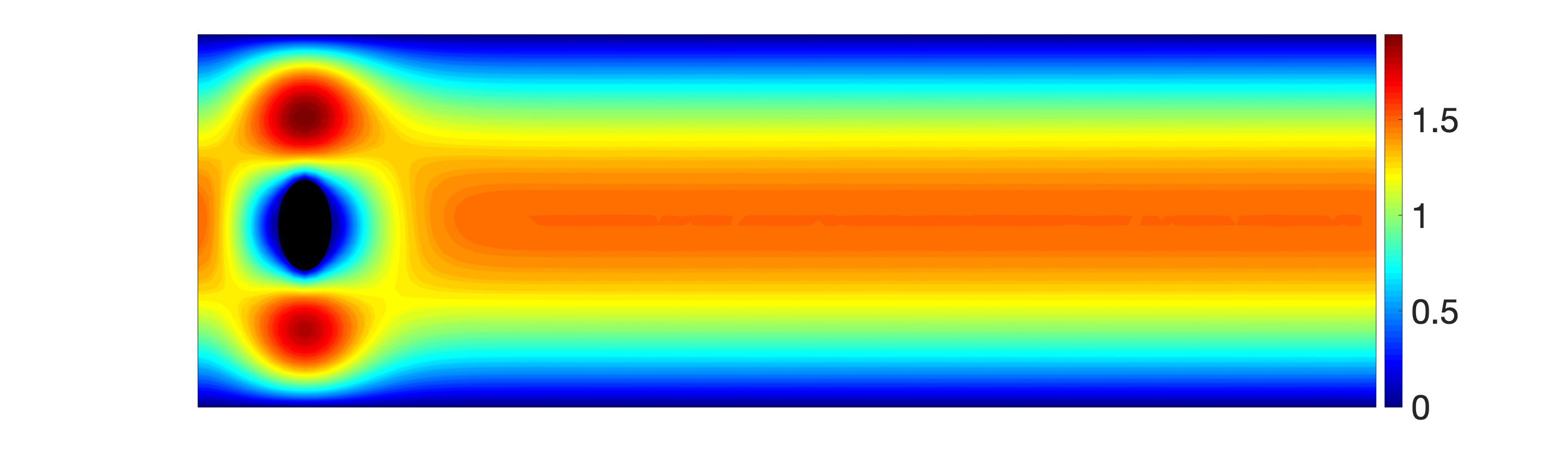}	 \\
	\includegraphics[width = 0.32\textwidth]{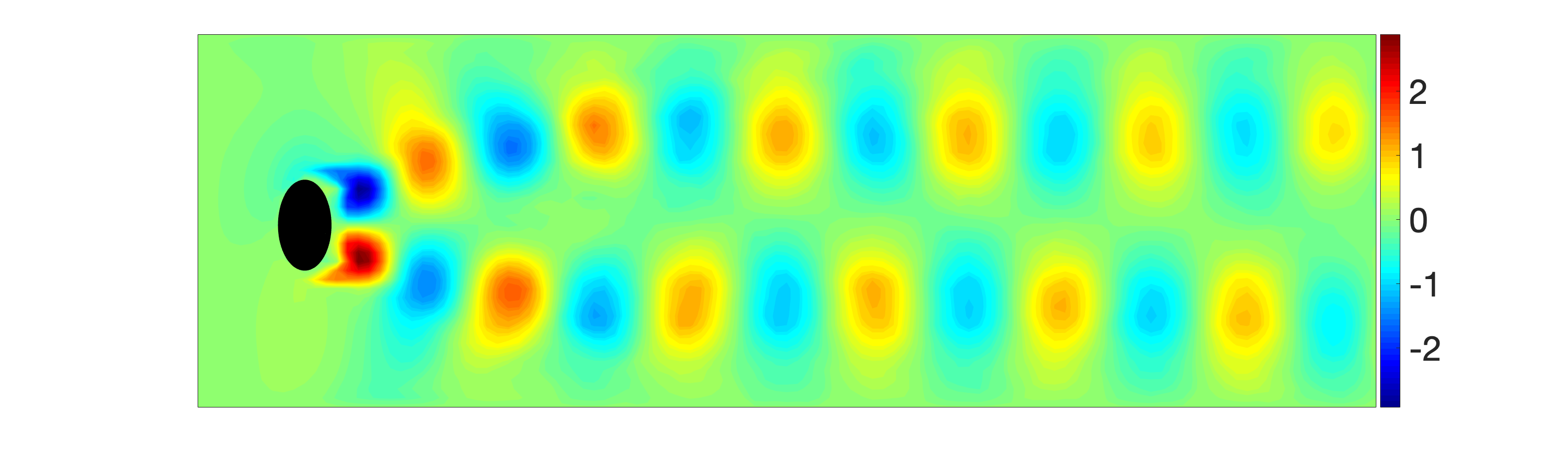}
	\includegraphics[width = 0.32\textwidth]{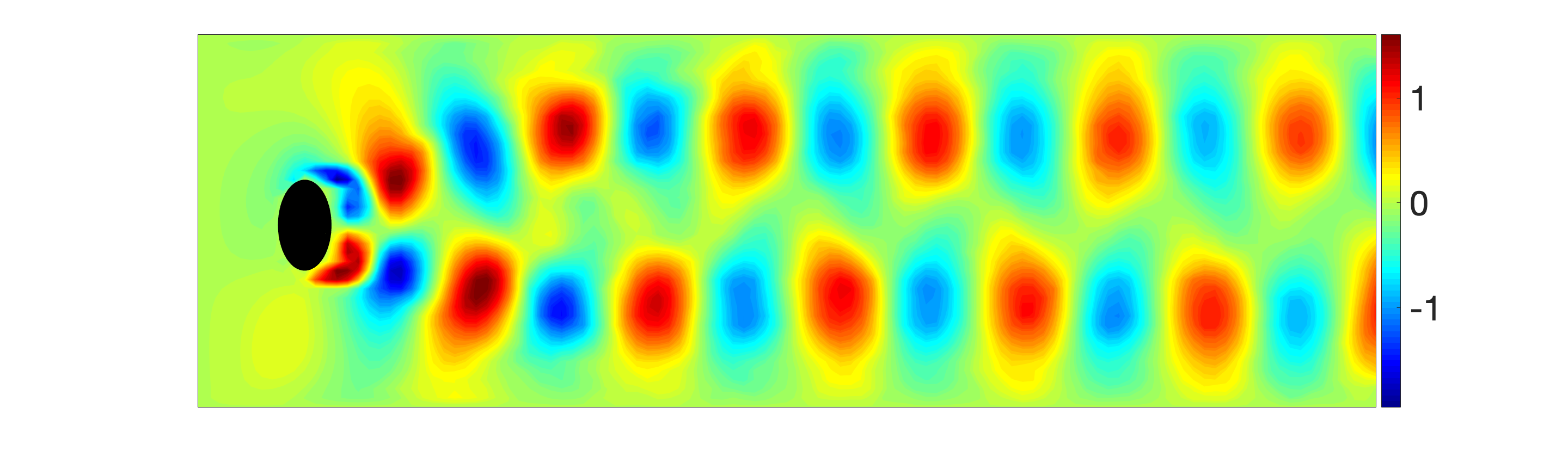}  
	\includegraphics[width = 0.32\textwidth]{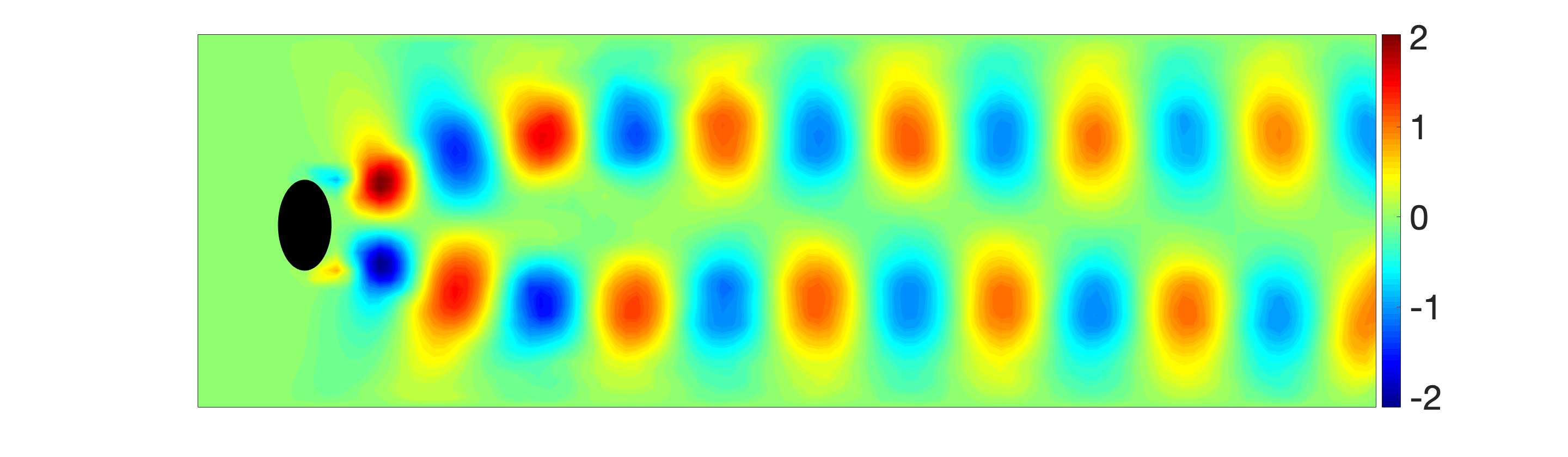}	 \\
	\includegraphics[width = 0.32\textwidth]{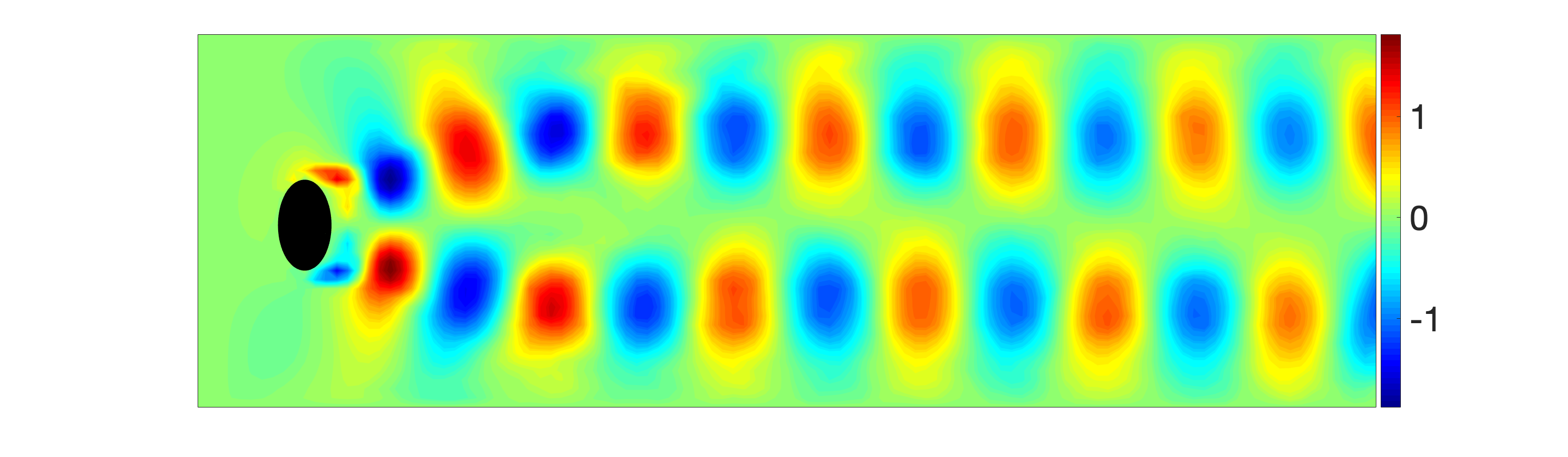}
	\includegraphics[width = 0.32\textwidth]{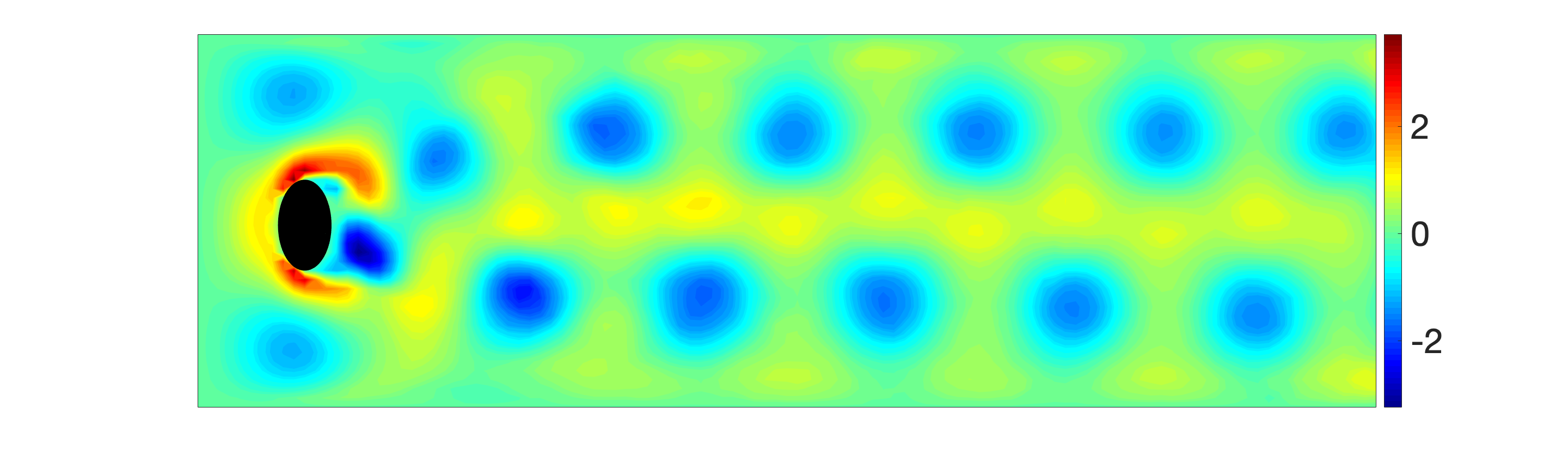}
	\includegraphics[width = 0.32\textwidth]{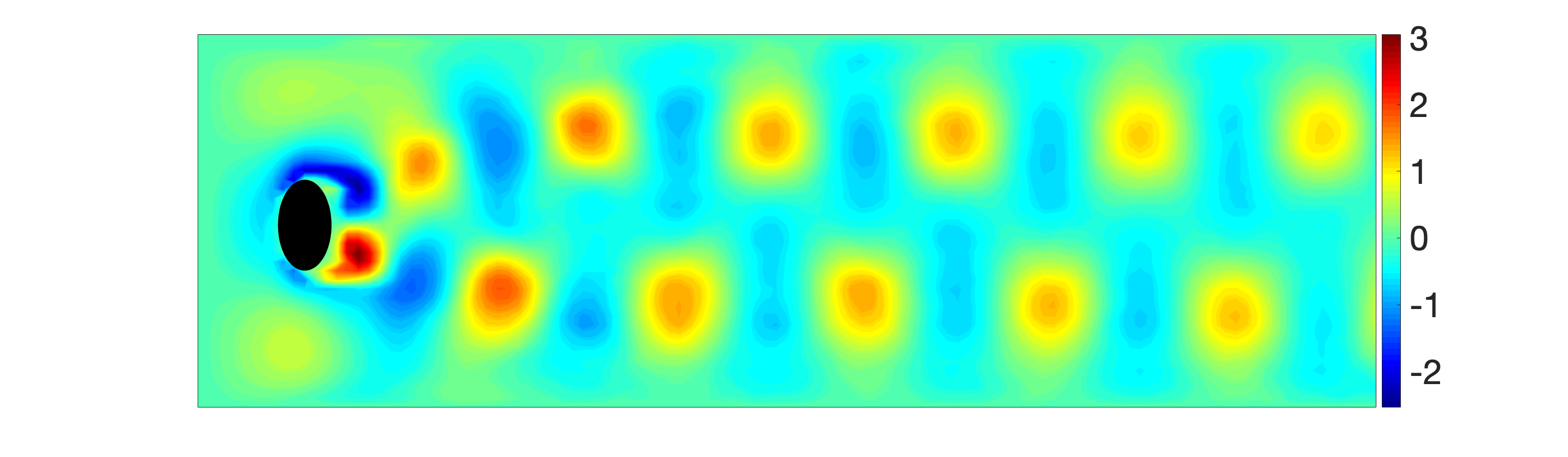}	  \\
	\includegraphics[width = 0.32\textwidth]{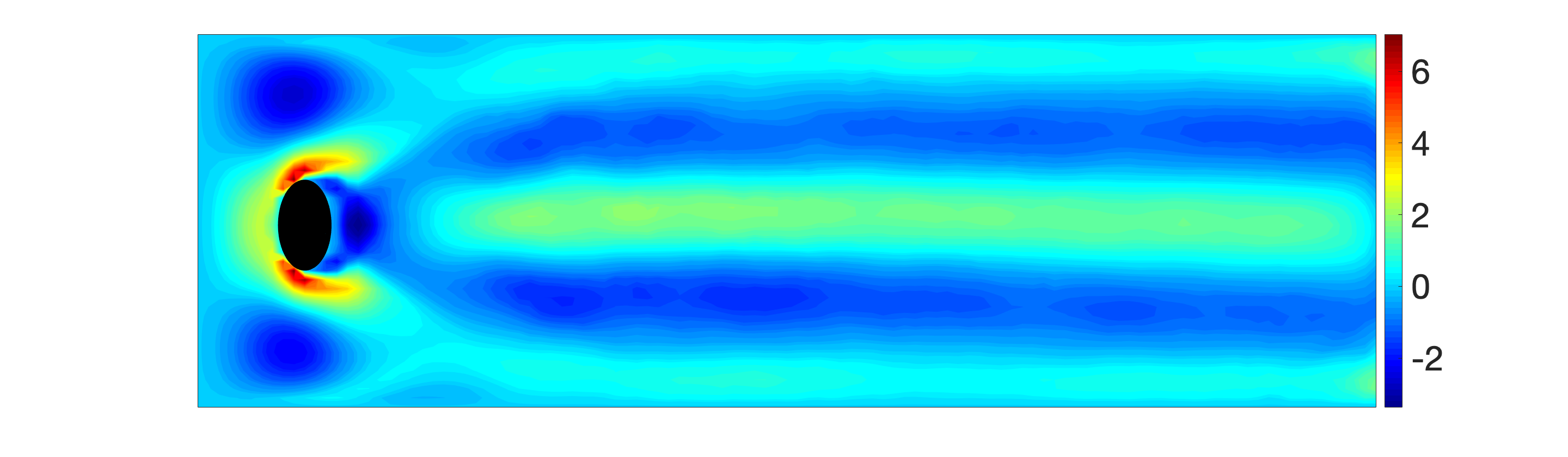}
	\includegraphics[width = .32\textwidth]{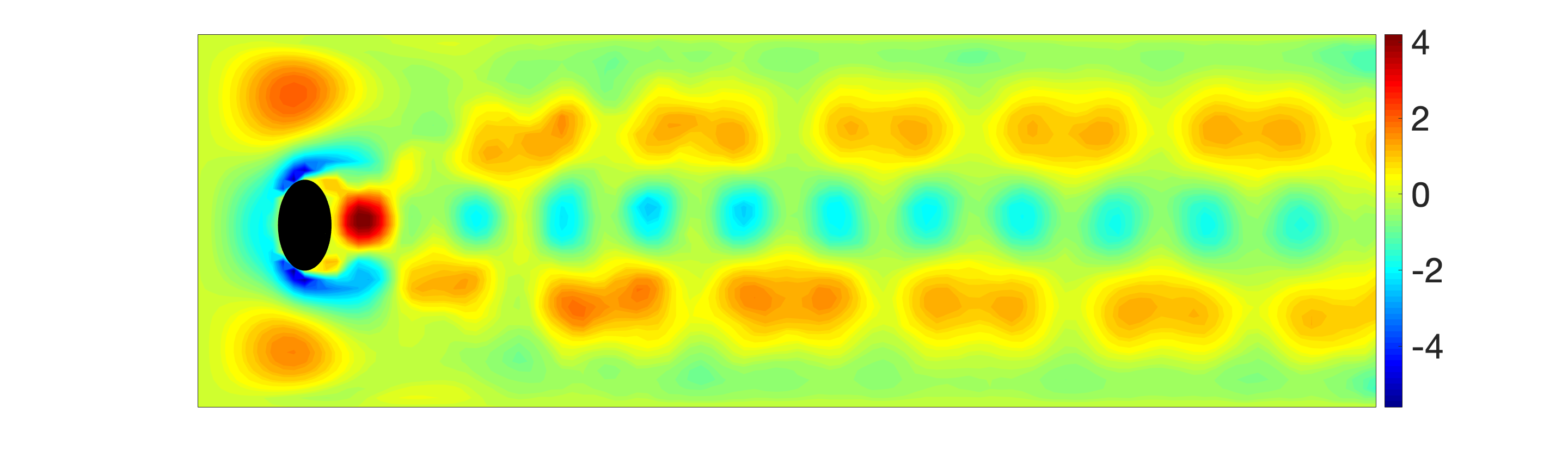}	
	\includegraphics[width = 0.32\textwidth]{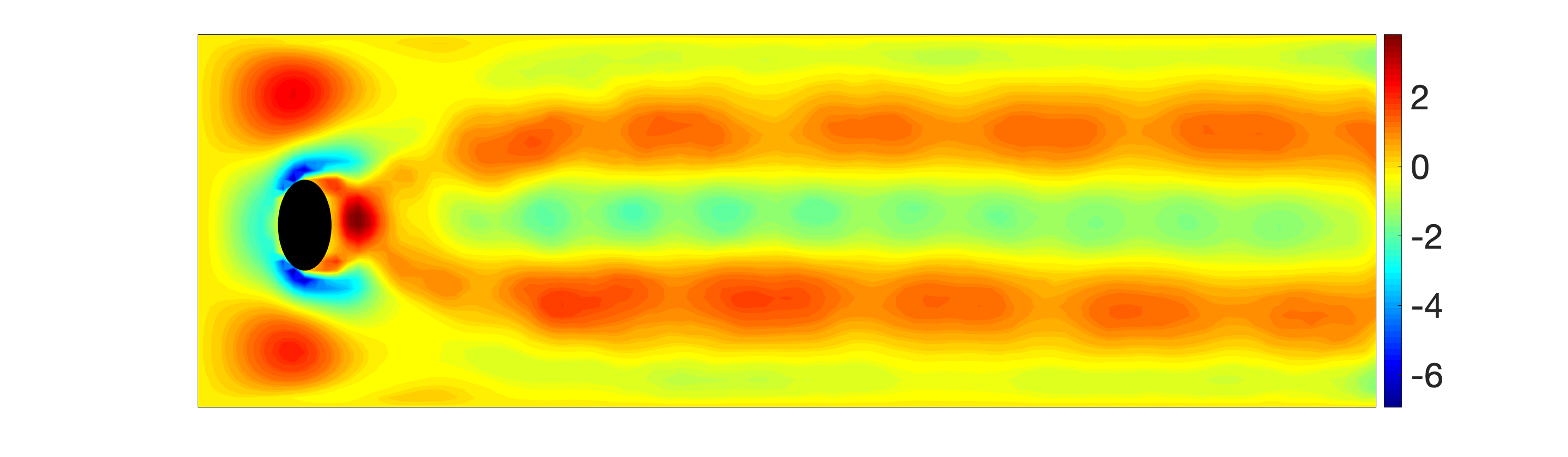}	 \\
	\includegraphics[width = 0.32\textwidth]{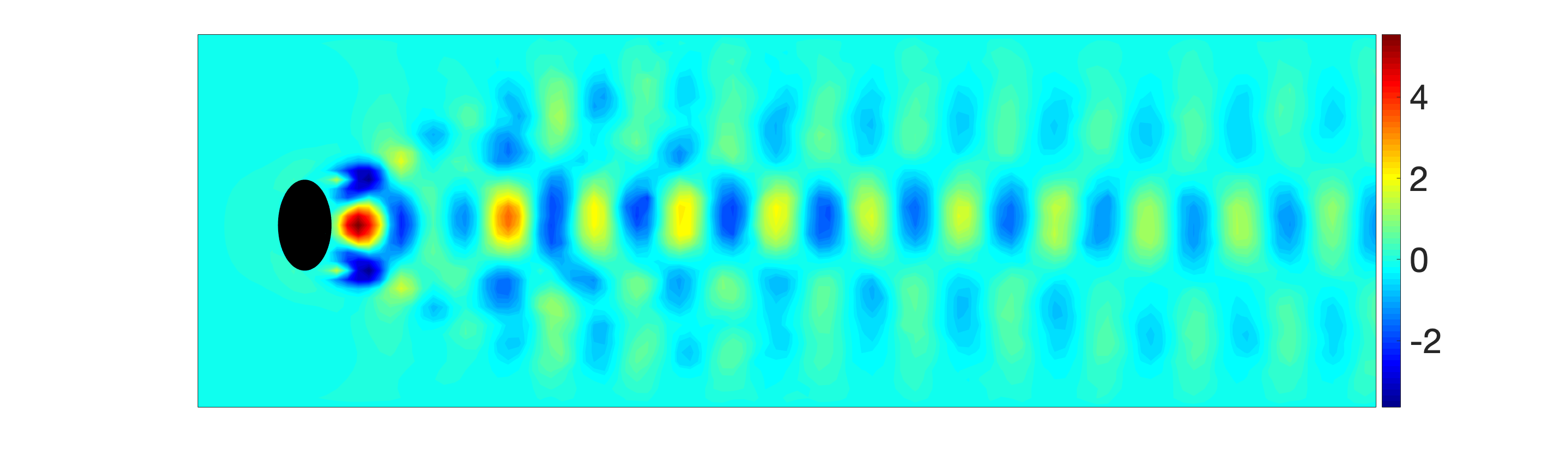}
	\includegraphics[width = .32\textwidth]{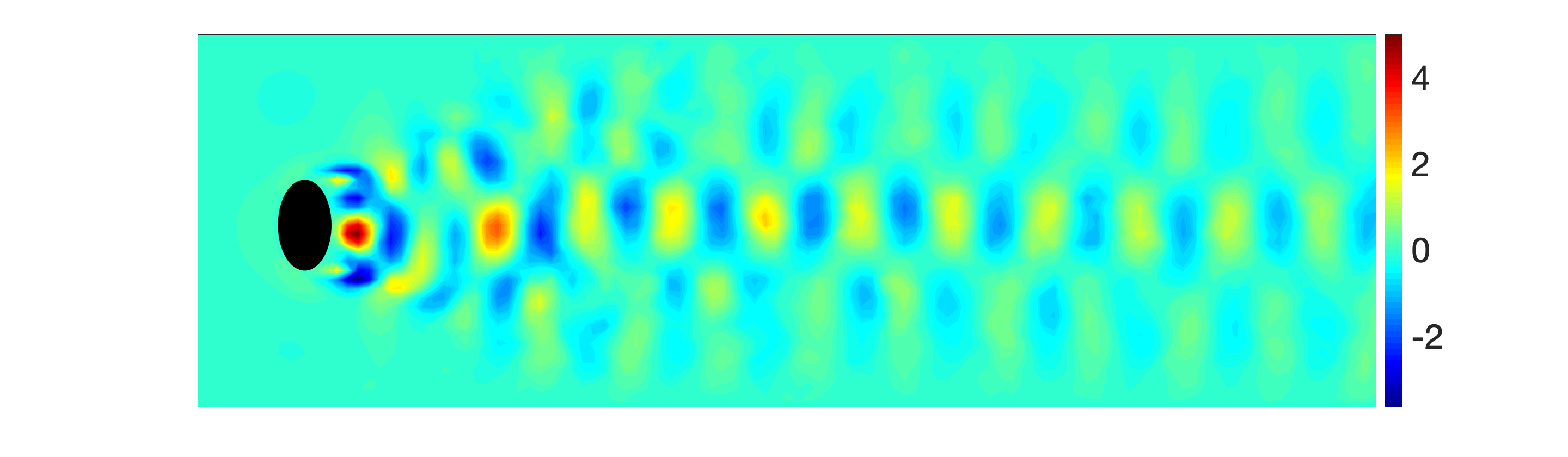}	
	\includegraphics[width = 0.32\textwidth]{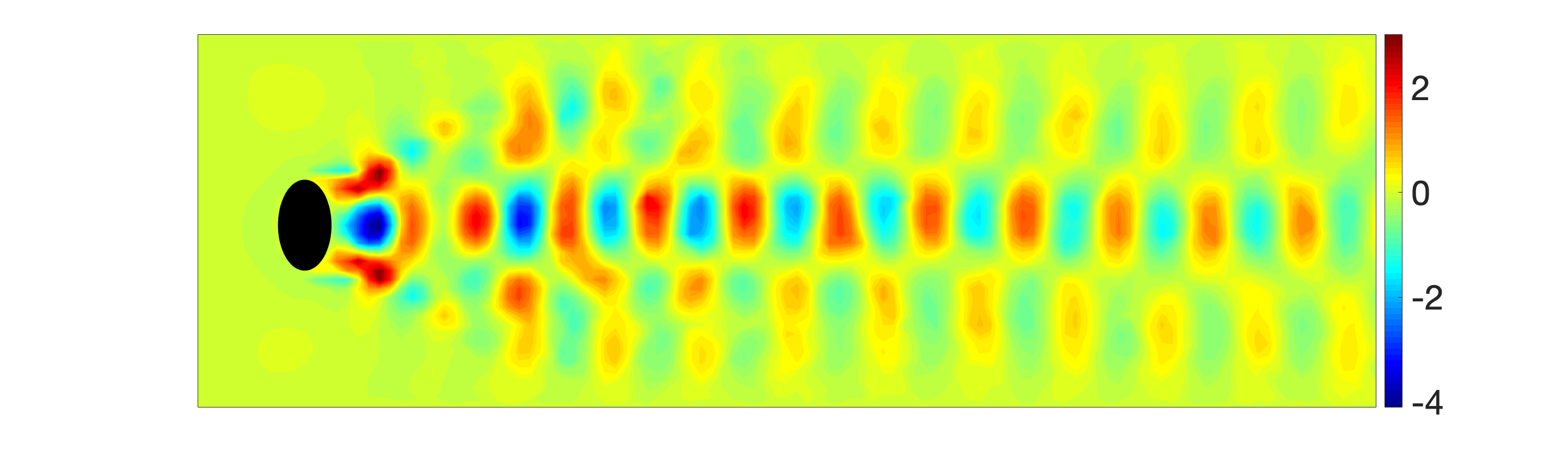}	
	\caption{\label{basisfncs} Pictured above are the first 5 basis functions generated by the ROM for first the full basis, then inaccurate bases 1 and 2, which both use less than one period of data to generate the basis. }
\end{figure}
In this section, we investigate the DA-ROM performance when the snapshots are inaccurate.
Specifically, we consider the same test as in Section 4.2, but now with only a small amount of data being used to build the ROM basis.  This is an important aspect of the ROM to investigate, because in practical applications complete data is generally not available, or the amount of data needed to sufficiently capture the behavior of the true solution is unknown. 

\begin{figure}[!ht]
	\centering
	Basis 1: \\
	\includegraphics[width = .68\textwidth, height = 1.39in]{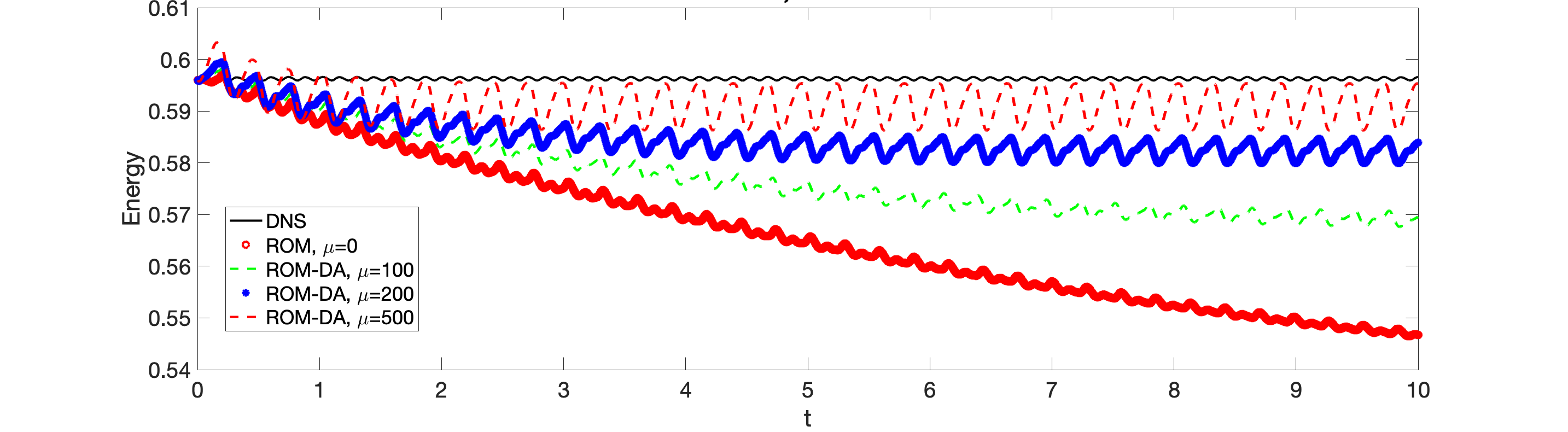}	 
	\includegraphics[width = .3\textwidth]{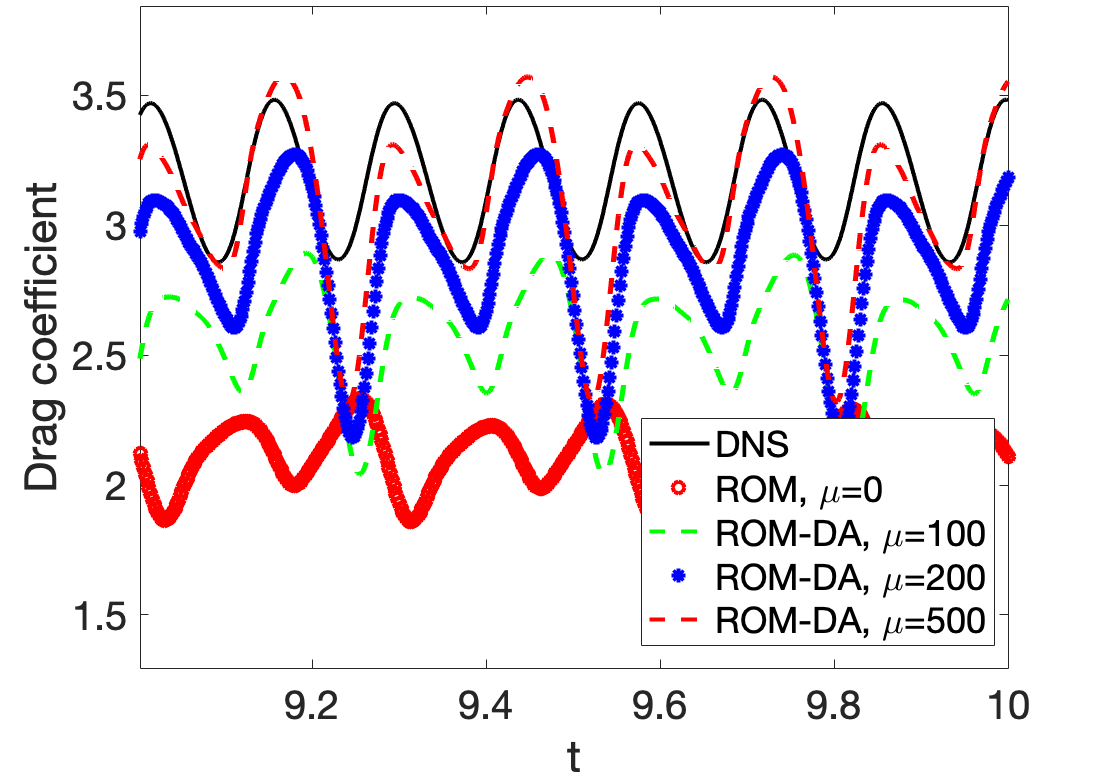} \\
	Basis 2: \\
	\includegraphics[width = .68\textwidth, height  = 1.39in]{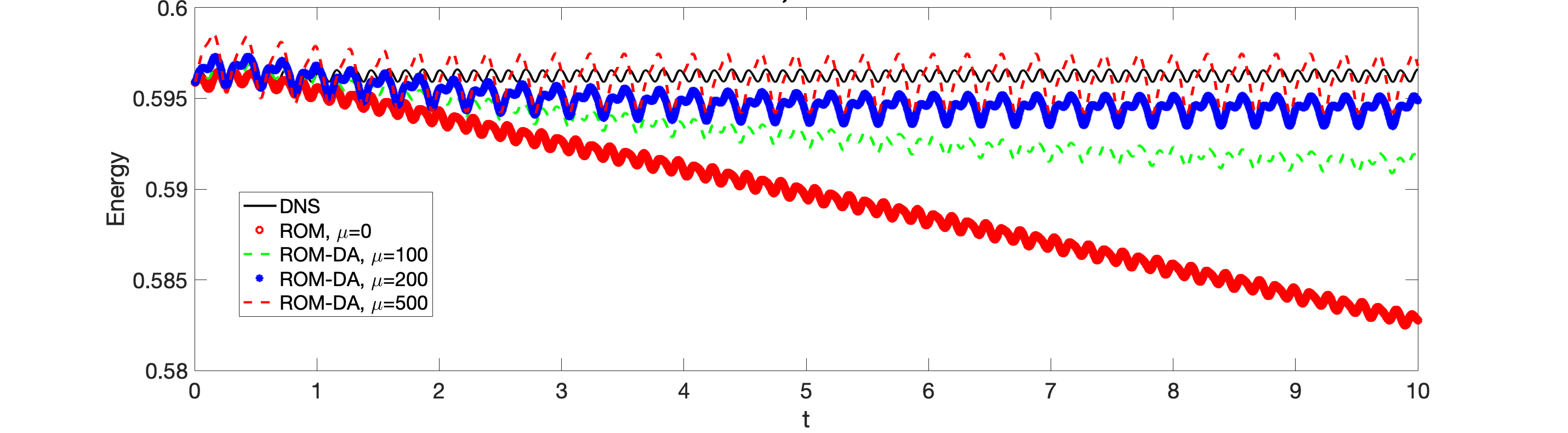}  
	\includegraphics[width = .3\textwidth]{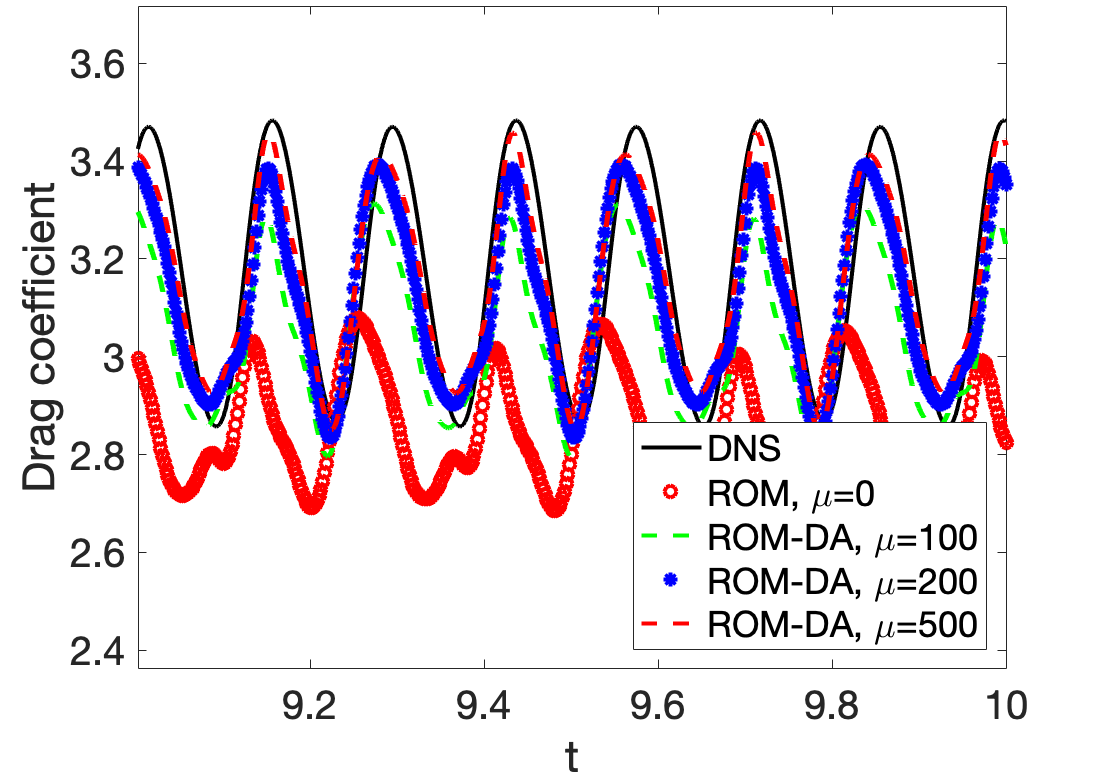}	
	\caption{\label{N9_re500_bb_plots} Energy and drag coefficient versus time plots with different values of $\mu$ for $Re=500$ using 8 modes and $H=\frac{2.2}{20}$. }
\end{figure}

\begin{figure}[!ht]
	\centering
	\hspace{0in} Basis 1 \\
	\includegraphics[width = .68\textwidth, height = 1.42in]{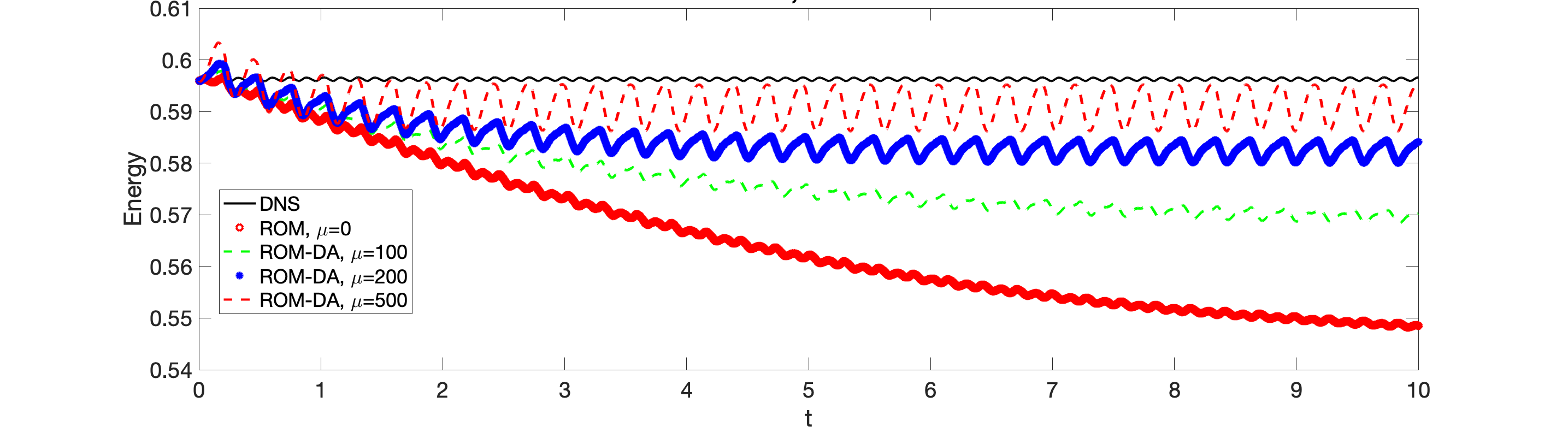}	
	\includegraphics[width = .31\textwidth]{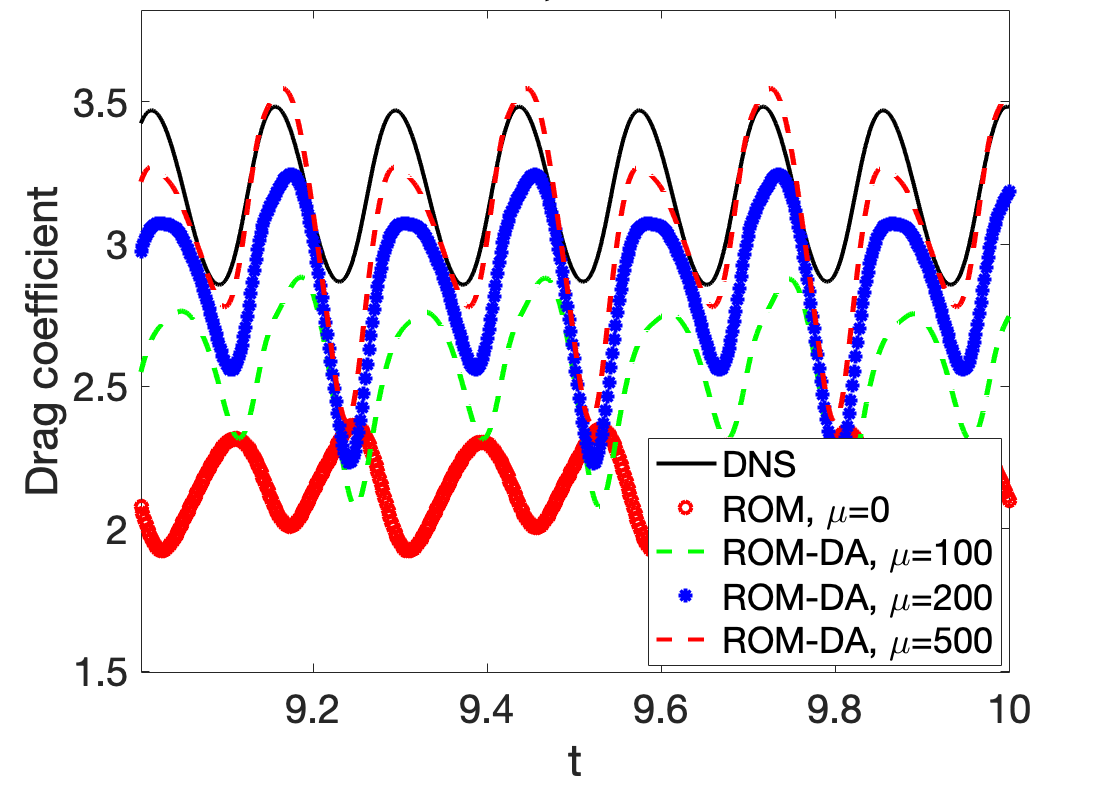} \\  
	\hspace{0in} Basis 2 \\
	\includegraphics[width = 0.68\textwidth, height = 1.42in]{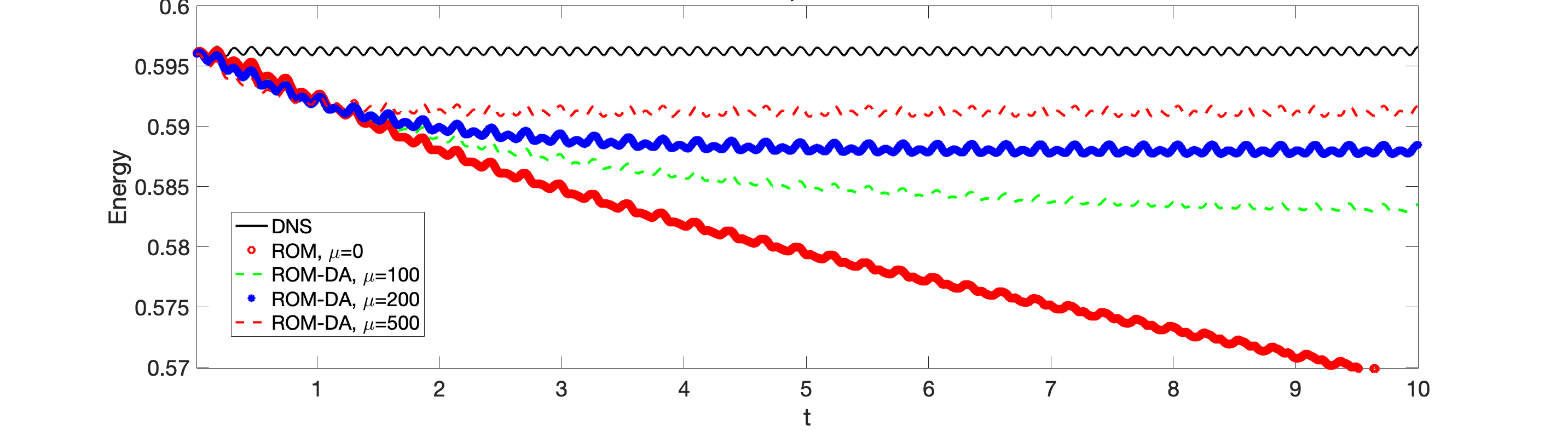}  	
	\includegraphics[width = .31\textwidth]{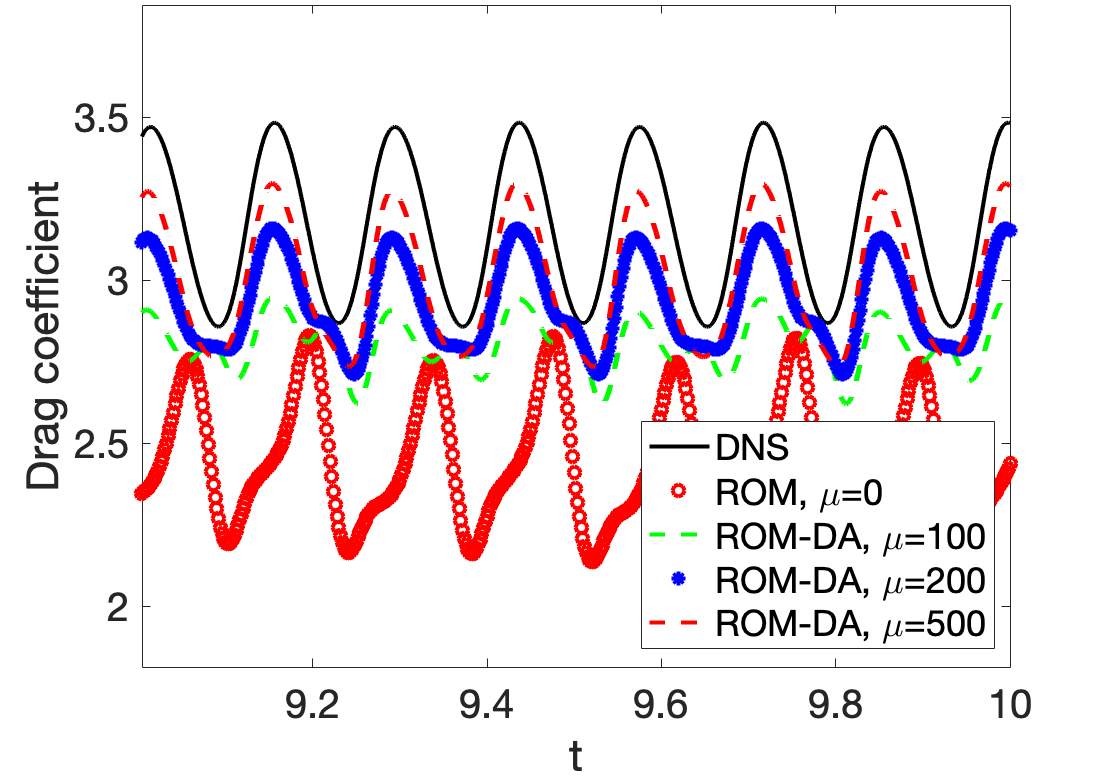}	
	\caption{\label{N13_re500_bb_plots} Energy and drag coefficient versus time plots with different values of $\mu$ for $Re=500$ using 12 modes and $H=\frac{2.2}{20}$. }
\end{figure}

We generated these inaccurate snapshots for $Re$=500 using less than one period of data: basis 1 used 64\% of one period of data while basis 2 used 84\%. See figure \ref{basisfncs} for the first five basis functions generated by the ROM; the basis functions for the full ROM are also included for comparison. 

In figure \ref{N9_re500_bb_plots}, we show the results of the DA-ROM using only 8 modes, with basis 1 and 2 defined above, and $\mu$ ranging from 100 to 500. 
DA significantly improves the accuracy of the ROM, and basis 2 does better at predicting the drag coefficient than basis 1. 

Figure \ref{N13_re500_bb_plots} shows energy and drag coefficient plots versus time using $N=12$ modes, and the nudging parameter $\mu$ is varied from 100 to 500. 
We see similar results as the case of using 8 modes; for both bases, DA significantly improves the accuracy of the ROM, compared to the ROM without DA ($\mu=0$), which becomes more and more inaccurate as time goes on. 
Basis 2 is able to very accurately predict the drag coefficient. 

The results in this section suggest that DA can dramatically improve the accuracy of a ROM when insufficient data is available to build the ROM, which is the general case in practical applications. 
We also emphasize that the improvement in the DA-ROM accuracy over the standard ROM accuracy is significantly larger in the realistic case of inaccurate snapshot construction.  
Indeed, comparing figures~\ref{N9_re500_bb_plots} and~\ref{N13_re500_bb_plots} with figure~\ref{N8_re500_plots}, we notice that the absolute improvement in the DA-ROM is much larger in the former than in the latter (this could be clearly seen from the magnitude of the $y$-axis).


	\subsection{Adaptive Nudging}
	\label{sec:variable-param}
	
To further improve the accuracy of the DA-ROM solution, we also consider nudging that is adaptive in time. 
While the error estimate we prove guarantees convergence up to discretization error and ROM truncation error exponentially fast in time, it may not be sufficient to expect good numerical results.  In practice, the ROM truncation error is often quite large, and can make the error bounds be too large to guarantee accurate predictions, especially over long time intervals. We propose below an adaptive nudging technique that will help produce better results by forcing the DA-ROM predicted energy to be more accurate.	
	 
\subsubsection{Algorithm}

In this section, we propose to change $\mu$ adaptively in time, based on the accuracy of the energy prediction of the ROM as well as the sign of the contribution of the data assimilation term to the energy balance. The semi-discrete algorithm reads: Find $u_r \in X_r$ such that for all $v_r \in X_r$, 
\begin{align}
((u_r)_t , v_r) + b^*(u_r, u_r, v_r) + \nu (\nabla u_r, \nabla v_r)  
 +  \mu (\Plh (u_r- u), \Plh v_r) = (f, v_r), \label{adapDAROMalg}
\end{align} 
with $v_0=P_r(u_0)$, and $\mu$ is the adaptive nudging parameter.  

  We begin the discussion with an energy estimate.  Choosing $\chi_r=u_r$ vanishes the nonlinear term, and after bounding the forcing term in the usual way we obtain the energy estimate
\[
\frac{d}{dt} \| u_r \|^2 + \nu \| \nabla u_r \|^2 + \mu \left( \| I_H(u_r) \|^2 - \| I_H(u) \|^2 + \| I_H(u_r - u ) \|^2 \right) \le \nu^{-1} \| f \|_{-1}.
\]
We assume this estimate is sharp in the following analysis, and that we know $\| u(t^n) \|$ in addition to $I_H(u)(t^n)$.

The adaptive strategy is to adjust $\mu$ so the contribution of the data assimilation term removes dissipation if the ROM-DA energy is too small, and adds dissipation if the energy is too large.  We use the term dissipation loosely, since here we refer to dissipation from the DA term only meaning that it adds positivity to the left hand side of the energy estimate.  Now after step $n$ we can calculate (1) the DA-ROM energy $\frac12 \| u_r^n \|^2$ and the true energy $\frac12 \| u(t^n) \|^2$; and (2) the sign of the contribution of the data assimilation term (DAT):
\[DAT:=  \| I_H(u_r^n) \|^2 - \| I_H(u)(t^n) \|^2 + \| I_H(u_r - u )(t^n) \|^2 .\] 
With this information, we check the energy error to see if it is too high (or too low), and if so, then add dissipation by increasing $\mu$ if $DAT>0$ and decreasing $\mu$ otherwise; or do the opposite to decrease dissipation.

How often to adjust $\mu$, and by how much each time, are interesting questions.  In our numerical tests below, we checked the value of $DAT$ every 10 time steps, since there is some calculation cost involved, and changed $\mu$ by $\pm 1$ each time, as large sudden changes in $\mu$ gave bad results.

\subsubsection{Numerical Results}

We follow the same problem set up outlined in Section 4.2 (again using the full ROM basis), but now choosing $\mu$ adaptively in time. 
We note that, in addition to the Reynolds number we considered in the previous numerical experiments (i.e., $Re=500$), we also consider $Re=1000$.
Figures \ref{N8_re500_adap_plots} and \ref{difference1000} show the energy and drag plots for the DA-ROM algorithm with the adaptive nudging described above, and for constant $\mu$, for no DA.  For both $Re$,  the adaptive DA-ROM yields  the most accurate results, outperforming the ROM without DA, and the DA-ROM with a constant $\mu$. Also included are plots of the $\mu$ values chosen by the algorithm at each timestep. We observe that the behavior of the values of $\mu$ is similar to that of the plots of $DAT$ in the figures. 

\begin{figure}[!ht]
	\centering
	\includegraphics[width = 0.68\textwidth, height = 1.63in]{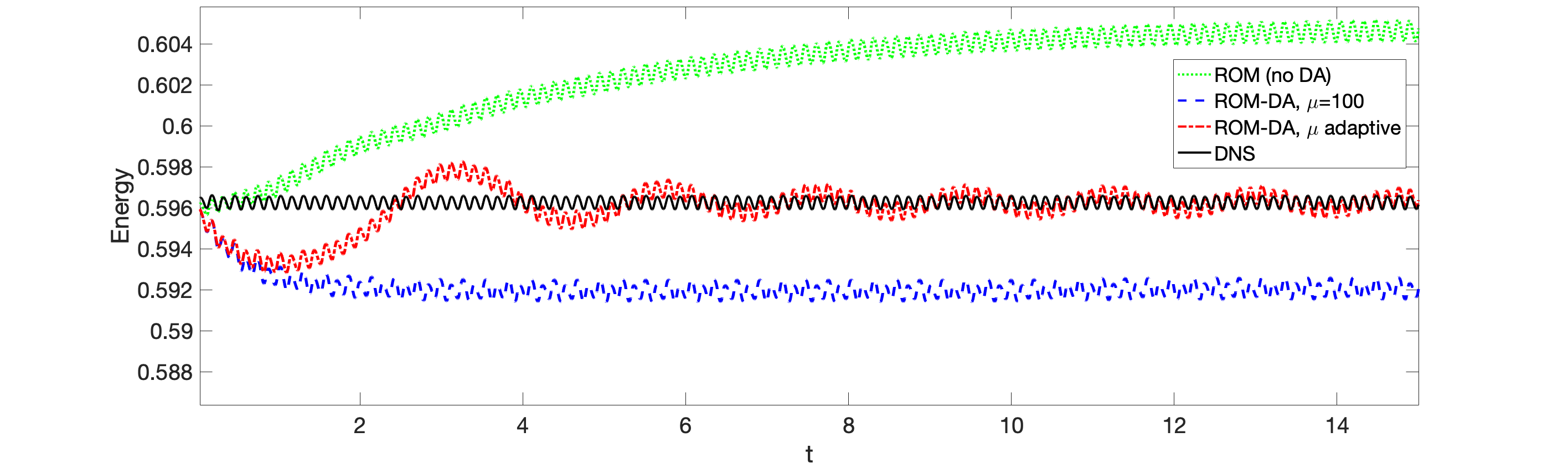}	 
	\includegraphics[width = .3\textwidth]{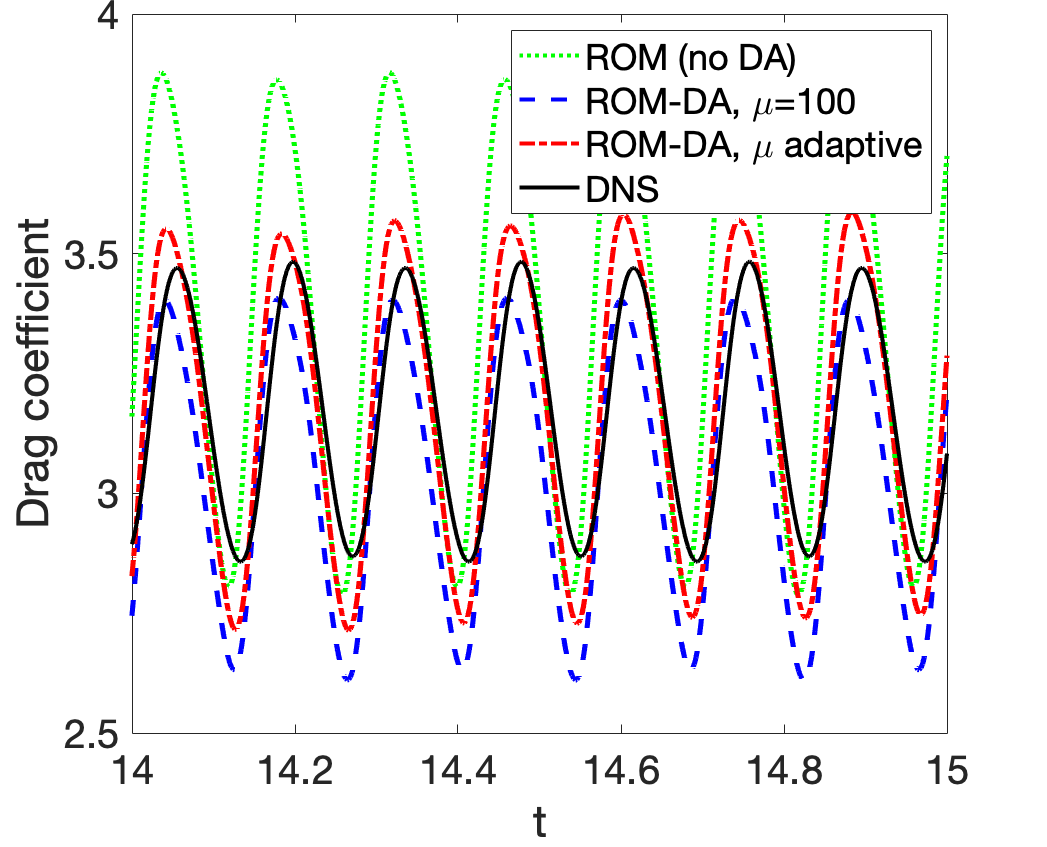}	 \\
	\includegraphics[width = .4\textwidth]{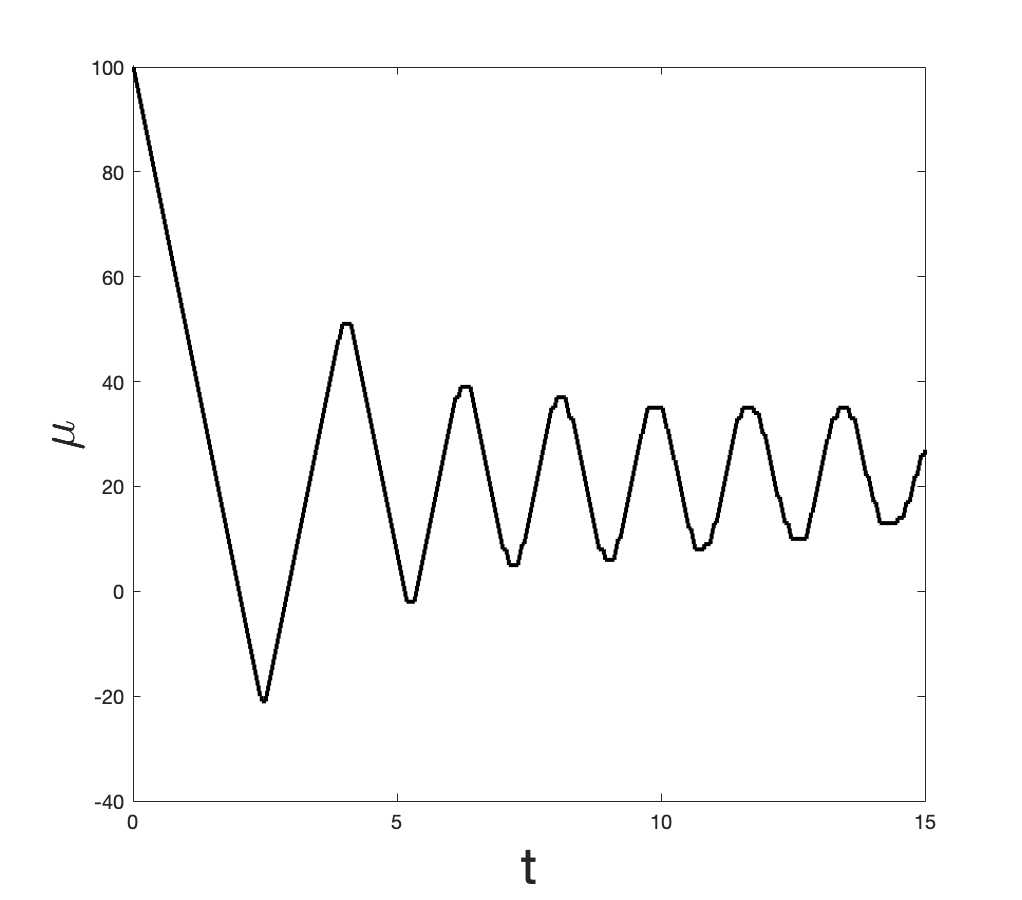}	
	\includegraphics[width = .4\textwidth]{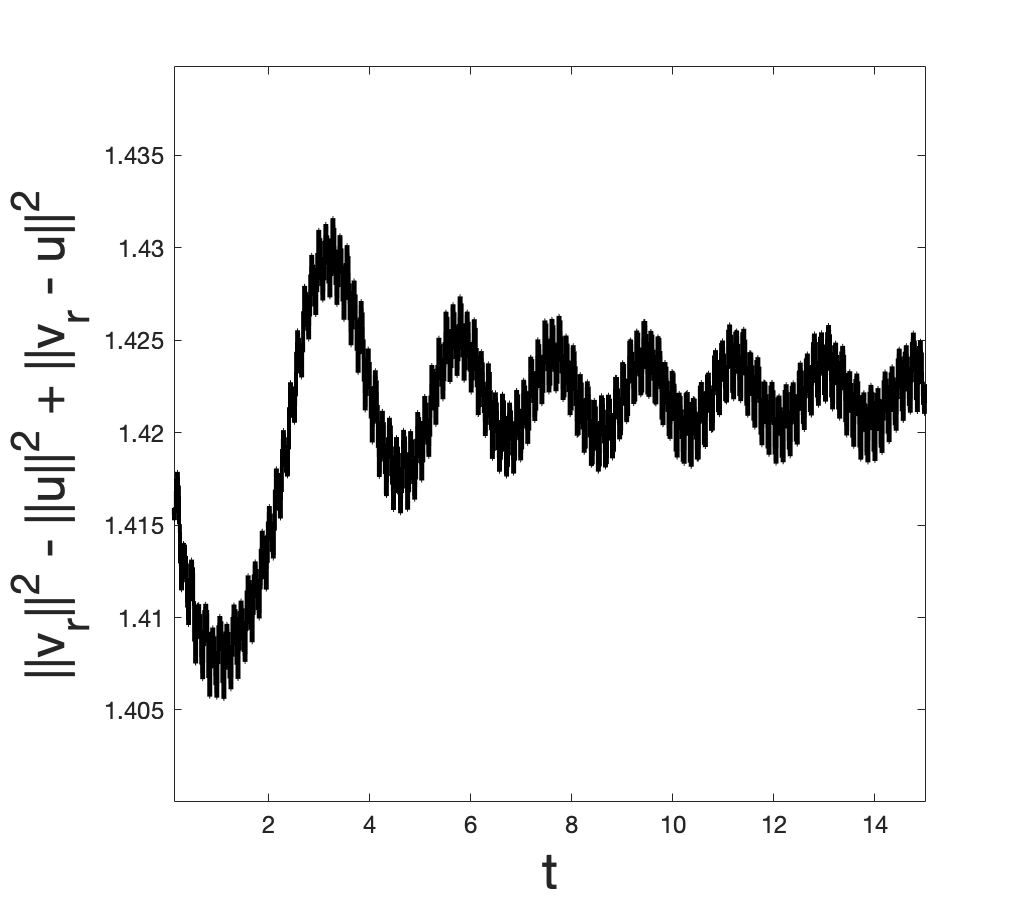}	
	\caption{\label{N8_re500_adap_plots} Energy and drag coefficient versus time for  $Re=500$ DA-ROM with different choices of $\mu$, with $N=8$ modes and $H=\frac{2.2}{20}$. Also included are the optimal choices of $\mu$ and the energy terms versus time, for the adaptive $\mu$ simulation.}
\end{figure}
\begin{figure}[!ht]
	\centering
	\includegraphics[width = .68\textwidth,height=.3\textwidth]{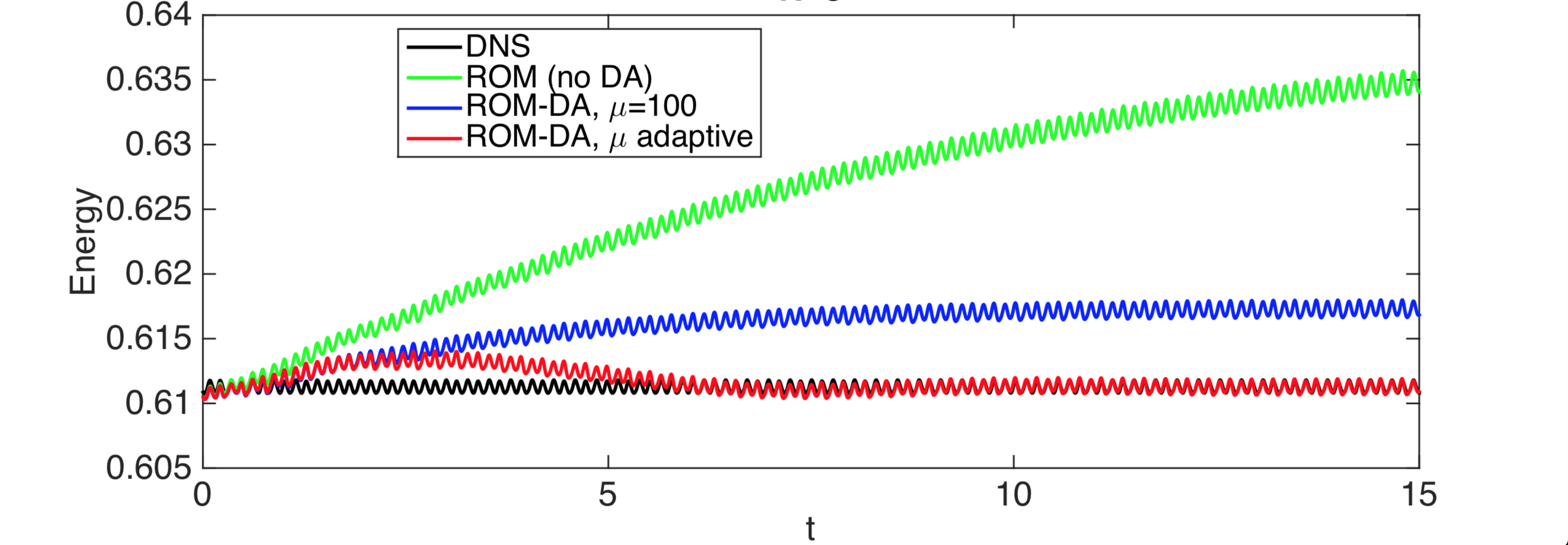}
	\includegraphics[width = .3\textwidth,height=.3\textwidth]{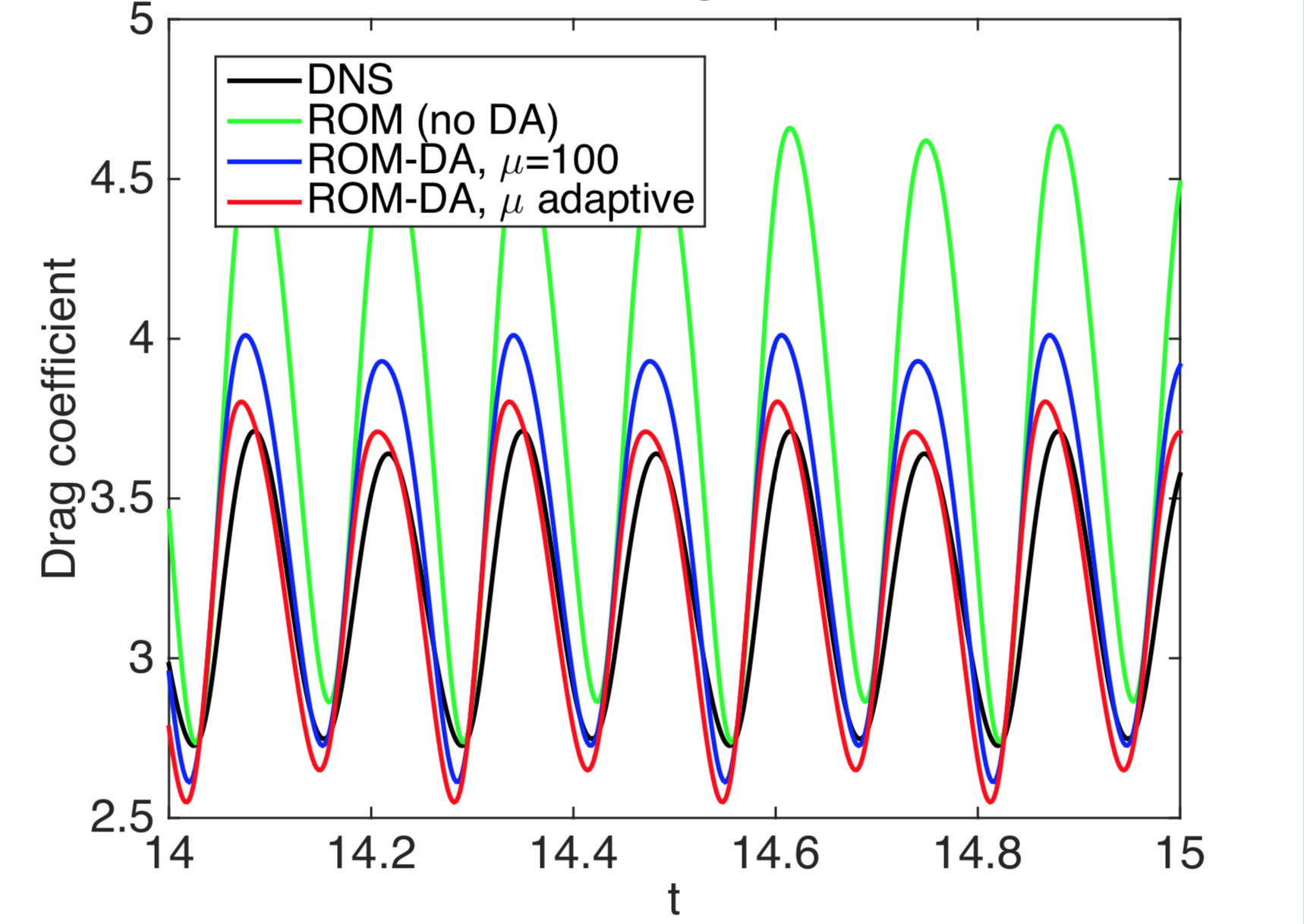}	\\
	\includegraphics[width = .4\textwidth,height=.3\textwidth]{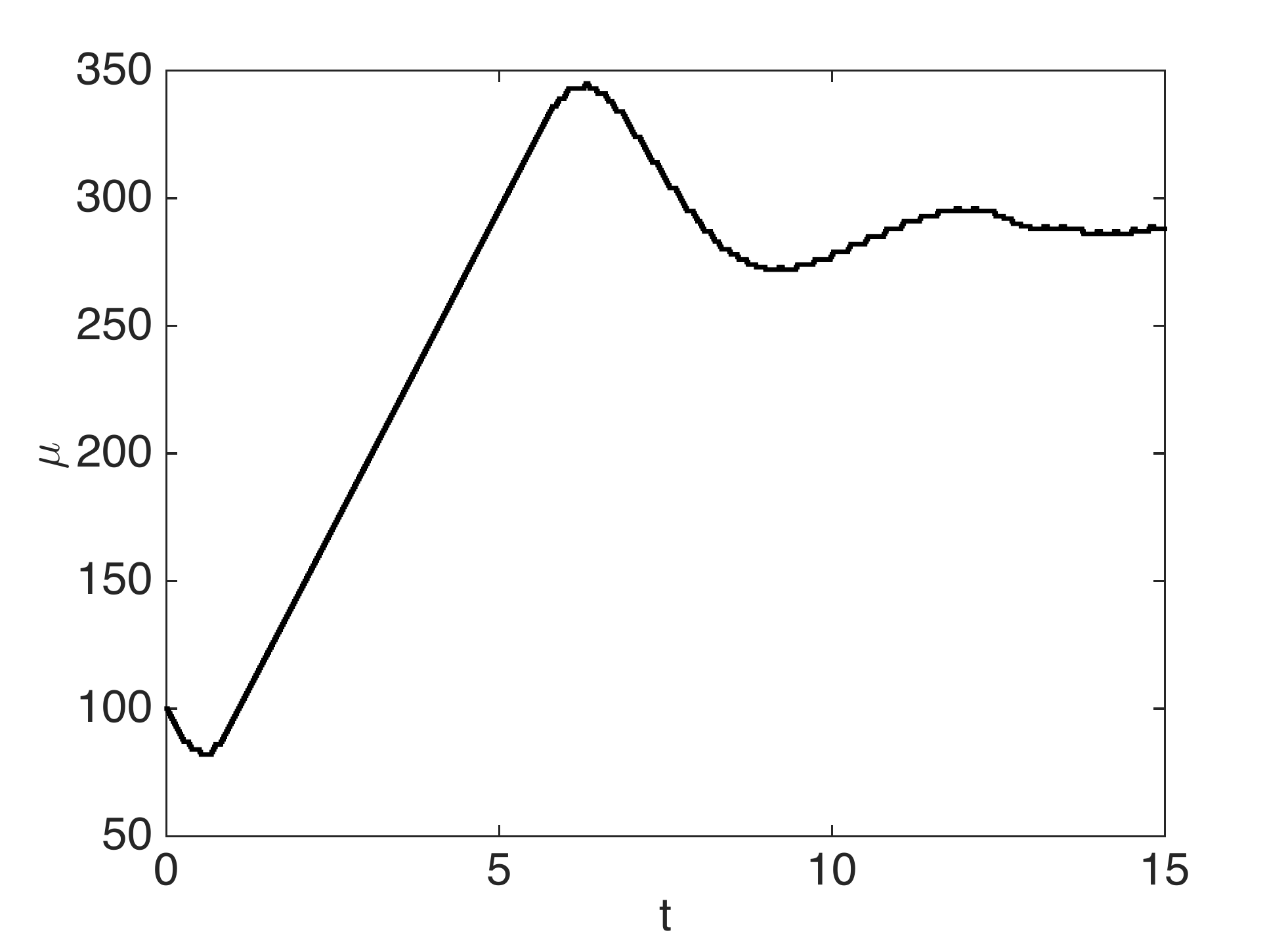}	
	\includegraphics[width = .4\textwidth,height=.3\textwidth]{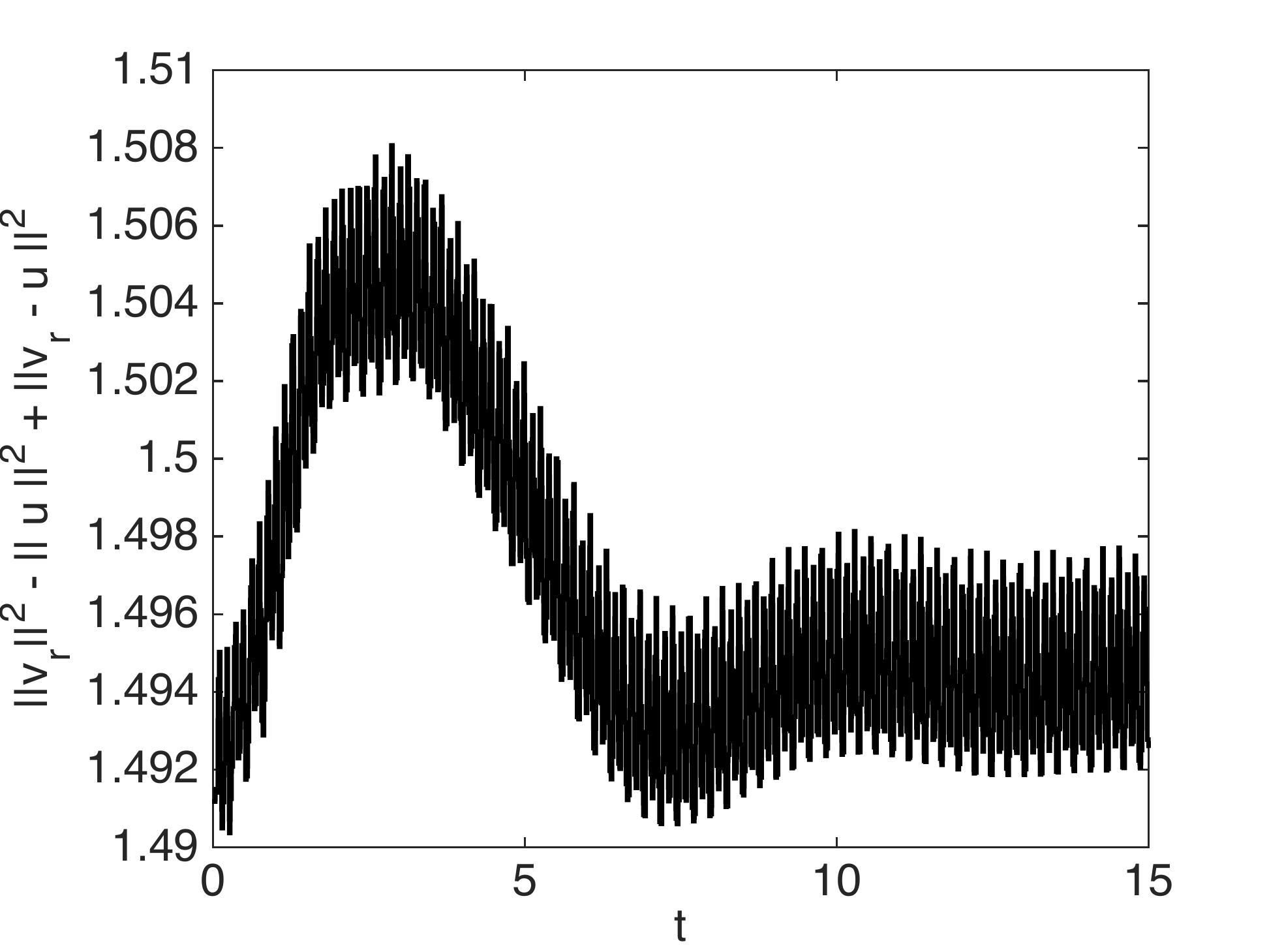}		
	\caption{\label{difference1000} Energy and drag coefficients versus time for  $Re=1000$ DA-ROM with different choices of $\mu$, with $N=8$ modes and $H=\frac{2.2}{20}$.  Shown at the bottom is $\mu$ and the contribution of the DA term versus time, for the adaptive $\mu$ simulation.}
\end{figure}

	\section{Conclusions}
	\label{sec:conclusions}
	
	In this paper, we put forth a new data assimilation reduced order model (DA-ROM) for fluid flows.
	The new DA-ROM adds to the standard ROM a feedback control term that nudges the ROM approximation towards the reference solution corresponding to the observed data.
The new DA-ROM's implementation is extremely simple:
The nudging term can be implemented into existing codes completely at the linear algebraic level, without any changes to the rest of the discretization.
	The nudging term dramatically increases the accuracy of the new DA-ROM by utilizing the available low-resolution data, without the need to increase the number of ROM basis functions.
	We proved that with a properly chosen nudging parameter, the new DA-ROM algorithm converges exponentially fast in time to the true solution, up to discretization and ROM truncation errors.  
	We also proposed a strategy for nudging adaptively in time, by adding or removing dissipation arising from the nudging to better match true solution energy. 
	Finally, we performed a numerical investigation of the new DA-ROM in the simulation of a 2D flow past a circular cylinder.
	The numerical results showed that the adaptive nudging DA-ROM significantly improves the long time ROM accuracy, especially when the snapshots used to construct the ROM are inaccurate, which is generally the case in realistic applications. 
	
We intend to pursue several research avenues.
First, we will investigate whether numerical analysis can help determine the optimal parameter in the adaptive nudging approach for the new DA-ROM. 
We also want to extend the numerical investigation of the DA-ROM to complex 3D flows.
Finally, we will examine whether using a spectral type of nudging in the DA-ROM instead of the current physical nudging yields better results.


	\bibliographystyle{abbrv}
	\bibliography{references,LariosBiblio,traian}

\begin{thebibliography}{10}

\bibitem{Azouani_Olson_Titi_2014}
A.~Azouani, E.~Olson, and E.~S. Titi.
\newblock Continuous data assimilation using general interpolant observables.
\newblock {\em J. Nonlinear Sci.}, 24(2):277--304, 2014.

\bibitem{balajewicz2016minimal}
M.~J. Balajewicz, I.~Tezaur, and E.~H. Dowell.
\newblock {Minimal subspace rotation on the Stiefel manifold for stabilization
  and enhancement of projection-based reduced order models for the compressible
  Navier--Stokes equations}.
\newblock {\em J. Comput. Phys.}, 321:224--241, 2016.

\bibitem{benosman2017learning}
M.~Benosman, J.~Borggaard, O.~San, and B.~Kramer.
\newblock {Learning-based robust stabilization for reduced-order models of 2D
  and 3D Boussinesq equations}.
\newblock {\em Appl. Math. Model.}, 49:162--181, 2017.

\bibitem{Bessaih_Olson_Titi_2015}
H.~Bessaih, E.~Olson, and E.~S. Titi.
\newblock Continuous data assimilation with stochastically noisy data.
\newblock {\em Nonlinearity}, 28(3):729--753, 2015.

\bibitem{Biswas_Martinez_2017}
A.~Biswas and V.~R. Martinez.
\newblock Higher-order synchronization for a data assimilation algorithm for
  the 2{D} {N}avier--{S}tokes equations.
\newblock {\em Nonlinear Anal. Real World Appl.}, 35:132--157, 2017.

\bibitem{cao2007reduced}
Y.~Cao, J.~Zhu, I.~M. Navon, and Z.~Luo.
\newblock A reduced-order approach to four-dimensional variational data
  assimilation using proper orthogonal decomposition.
\newblock {\em Int. J. Numer. Meth. Fluids}, 53(10):1571--1583, 2007.

\bibitem{carlberg2017galerkin}
K.~Carlberg, M.~Barone, and H.~Antil.
\newblock {Galerkin v. least-squares Petrov--Galerkin projection in nonlinear
  model reduction}.
\newblock {\em J. Comput. Phys.}, 330:693--734, 2017.

\bibitem{Farhat_Jolly_Titi_2015}
A.~Farhat, M.~S. Jolly, and E.~S. Titi.
\newblock Continuous data assimilation for the 2{D} {B}\'enard convection
  through velocity measurements alone.
\newblock {\em Phys. D}, 303:59--66, 2015.

\bibitem{Farhat_Lunasin_Titi_2016abridged}
A.~Farhat, E.~Lunasin, and E.~S. Titi.
\newblock Abridged continuous data assimilation for the 2{D} {N}avier--{S}tokes
  equations utilizing measurements of only one component of the velocity field.
\newblock {\em J. Math. Fluid Mech.}, 18(1):1--23, 2016.

\bibitem{Foias_Mondaini_Titi_2016}
C.~Foias, C.~F. Mondaini, and E.~S. Titi.
\newblock A discrete data assimilation scheme for the solutions of the
  two-dimensional {N}avier-{S}tokes equations and their statistics.
\newblock {\em SIAM J. Appl. Dyn. Syst.}, 15(4):2109--2142, 2016.

\bibitem{galletti2004low}
B.~Galletti, C.~H. Bruneau, L.~Zannetti, and A.~Iollo.
\newblock Low-order modelling of laminar flow regimes past a confined square
  cylinder.
\newblock {\em J. Fluid Mech.}, 503:161--170, 2004.

\bibitem{G99}
J.~L. Guermond.
\newblock Stabilization of {Galerkin} approximations of transport equations by
  subgrid modeling.
\newblock {\em M2AN, Math. Model. Numer. Anal.}, 33(6):1293--1316, 1999.

\bibitem{gunzburger2018leray}
M.~Gunzburger, T.~Iliescu, and M.~Schneier.
\newblock {A Leray regularized ensemble-proper orthogonal decomposition method
  for parameterized convection-dominated flows}.
\newblock {\em IMA J. Numer. Anal.}, 01 2019.

\bibitem{hesthaven2015certified}
J.~S. Hesthaven, G.~Rozza, and B.~Stamm.
\newblock {\em Certified Reduced Basis Methods for Parametrized Partial
  Differential Equations}.
\newblock Springer, 2015.

\bibitem{HLB96}
P.~Holmes, J.~L. Lumley, and G.~Berkooz.
\newblock {\em Turbulence, Coherent Structures, Dynamical Systems and
  Symmetry}.
\newblock Cambridge, 1996.

\bibitem{iliescu2014variational}
T.~Iliescu and Z.~Wang.
\newblock Variational multiscale proper orthogonal decomposition:
  {N}avier-{S}tokes equations.
\newblock {\em Num. Meth. P.D.E.s}, 30(2):641--663, 2014.

\bibitem{J04}
V.~John.
\newblock Reference values for drag and lift of a two dimensional
  time-dependent flow around a cylinder.
\newblock {\em Int. J. Numer. Methods Fluids}, 44:777--788, 2002.

\bibitem{Jolly_Martinez_Titi_2017}
M.~S. Jolly, V.~R. Martinez, and E.~S. Titi.
\newblock A data assimilation algorithm for the subcritical surface
  quasi-geostrophic equation.
\newblock {\em Adv. Nonlinear Stud.}, 17(1):167--192, 2017.

\bibitem{kaercher2018reduced}
M.~Kaercher, S.~Boyaval, M.~A. Grepl, and K.~Veroy.
\newblock Reduced basis approximation and a posteriori error bounds for 4d-var
  data assimilation.
\newblock {\em Optim. Eng.}, pages 1--33, 2018.

\bibitem{kalnay2003atmospheric}
E.~Kalnay.
\newblock {\em {Atmospheric modeling, data assimilation, and predictability}}.
\newblock Cambridge Univ Pr, 2003.

\bibitem{KV01}
K.~Kunisch and S.~Volkwein.
\newblock Galerkin proper orthogonal decomposition methods for parabolic
  problems.
\newblock {\em Numer. Math.}, 90(1):117--148, 2001.

\bibitem{Larios_Pei_2017_KSE_DA_NL}
A.~Larios and Y.~Pei.
\newblock Nonlinear continuous data assimilation.
\newblock (submitted) arXiv:1703.03546.

\bibitem{LRZ18}
A.~Larios, L.~Rebholz, and C.~Zerfas.
\newblock Global in time stability and accuracy of {IMEX-FEM} data assimilation
  schemes for {Navier-Stokes} equations.
\newblock {\em Comput. Methods Appl. Mech. Engrg.,}, 345:1077--1093, 2019.

\bibitem{layton2008introduction}
W.~J. Layton.
\newblock {\em Introduction to the numerical analysis of incompressible viscous
  flows}, volume~6.
\newblock Society for Industrial and Applied Mathematics (SIAM), 2008.

\bibitem{maday2015parameterized}
Y.~Maday, A.~T. Patera, J.~D. Penn, and M.~Yano.
\newblock A parameterized-background data-weak approach to variational data
  assimilation: formulation, analysis, and application to acoustics.
\newblock {\em Int. J. Num. Meth. Engng.}, 102(5):933--965, 2015.

\bibitem{Markowich_Titi_Trabelsi_2016}
P.~A. Markowich, E.~S. Titi, and S.~Trabelsi.
\newblock Continuous data assimilation for the three-dimensional
  {B}rinkman--{F}orchheimer-extended {D}arcy model.
\newblock {\em Nonlinearity}, 29(4):1292, 2016.

\bibitem{mohebujjaman2019physically}
M.~Mohebujjaman, L.~G. Rebholz, and T.~Iliescu.
\newblock Physically-constrained data-driven correction for reduced order
  modeling of fluid flows.
\newblock {\em Int. J. Num. Meth. Fluids}, 89(3):103--122, 2019.

\bibitem{mohebujjaman2017energy}
M.~Mohebujjaman, L.~G. Rebholz, X.~Xie, and T.~Iliescu.
\newblock Energy balance and mass conservation in reduced order models of fluid
  flows.
\newblock {\em J. Comput. Phys.}, 346:262--277, 2017.

\bibitem{osth2014need}
J.~{\"O}sth, B.~R. Noack, S.~Krajnovi{\'c}, D.~Barros, and J.~Bor{\'e}e.
\newblock On the need for a nonlinear subscale turbulence term in {POD} models
  as exemplified for a high-{R}eynolds-number flow over an {A}hmed body.
\newblock {\em J. Fluid Mech.}, 747:518--544, 2014.

\bibitem{peherstorfer2016data}
B.~Peherstorfer and K.~Willcox.
\newblock Data-driven operator inference for nonintrusive projection-based
  model reduction.
\newblock {\em Comput. Methods Appl. Mech. Engrg.}, 306:196--215, 2016.

\bibitem{P97}
A.~Prohl.
\newblock {\em Projection and quasi-compressibility methods for solving the
  incompressible {Navier-Stokes} equations}.
\newblock Teubner-Verlag, Stuttgart, 1997.

\bibitem{protas2015optimal}
B.~Protas, B.~R. Noack, and J.~{\"O}sth.
\newblock Optimal nonlinear eddy viscosity in {G}alerkin models of turbulent
  flows.
\newblock {\em J. Fluid Mech.}, 766:337--367, 2015.

\bibitem{quarteroni2015reduced}
A.~Quarteroni, A.~Manzoni, and F.~Negri.
\newblock {\em Reduced Basis Methods for Partial Differential Equations: An
  Introduction}, volume~92.
\newblock Springer, 2015.

\bibitem{RX18}
L.~G. {Rebholz} and C.~{Zerfas}.
\newblock {Simple and efficient continuous data assimilation of evolution
  equations via algebraic nudging}.
\newblock {\em arXiv e-prints}, page arXiv:1810.03512, Oct. 2018.

\bibitem{rowley2004model}
C.~W. Rowley, T.~Colonius, and R.~M. Murray.
\newblock Model reduction for compressible flows using {POD} and {G}alerkin
  projection.
\newblock {\em Phys. D}, 189(1):115--129, 2004.

\bibitem{san2015principal}
O.~San and J.~Borggaard.
\newblock {Principal interval decomposition framework for POD reduced-order
  modeling of convective Boussinesq flows}.
\newblock {\em Int. J. Num. Meth. Fluids}, 78(1):37--62, 2015.

\bibitem{san2015stabilized}
O.~San and T.~Iliescu.
\newblock A stabilized proper orthogonal decomposition reduced-order model for
  large scale quasigeostrophic ocean circulation.
\newblock {\em Adv. Comput. Math.}, pages 1289--1319, 2015.

\bibitem{ST96}
M.~Sch\"afer and S.~Turek.
\newblock The benchmark problem `flow around a cylinder' flow simulation with
  high performance computer {II}.
\newblock {\em Notes on Numerical Fluid Mechanics}, 52:547--566, 1996.

\bibitem{stefanescu2015pod}
R.~{\c{S}}tef{\u{a}}nescu, A.~Sandu, and I.~M. Navon.
\newblock {POD/DEIM reduced-order strategies for efficient four dimensional
  variational data assimilation}.
\newblock {\em J. Comput. Phys.}, 295:569--595, 2015.

\bibitem{XMRT18}
X.~Xie, M.~Mohebujjaman, L.~Rebholz, and T.~Iliescu.
\newblock Data-driven filtered reduced order modeling of fluid flows.
\newblock {\em SIAM Journal on Scientific Computing}, 40(3):B834--B857, 2018.

\bibitem{xie2018lagrangian}
X.~Xie, P.~J. Nolan, S.~D. Ross, and T.~Iliescu.
\newblock Lagrangian data-driven reduced order modeling of finite time
  {L}yapunov exponents.
\newblock 2018.
\newblock available as arXiv preprint, \url{http://arxiv.org/abs/1808.05635}.

\end{thebibliography}

\end{document}